\documentclass[reqno]{amsart}
\usepackage{amsmath,amsfonts,amssymb,color}
\definecolor{col4}{rgb}{0.7,0,0.3}
\newtheorem{teor}{Theorem}[section]

\newtheorem{prop}[teor]{Proposition}

\newtheorem{coro}[teor]{Corollary}
\theoremstyle{definition}
\newtheorem{defi}[teor]{Definition}

\newtheorem{nota}[teor]{Remark}

\numberwithin{equation}{section}
\newcommand{\N}{\mathbb{N}}

\newcommand{\R}{\mathbb{R}}

\newcommand{\A}{\mathbb{A}}
\newcommand{\B}{\mathbb{B}}
\newcommand{\F}{\mathbb{F}}

\newcommand{\dd}{\textsf{d}}

\newcommand{\wit}{\widetilde}

\newcommand{\lsm}{\left[\!\begin{smallmatrix}}
\newcommand{\rsm}{\end{smallmatrix}\!\right]}
\newcommand{\nbd}{\nobreakdash}
\newcommand{\des}{\displaystyle}

\DeclareMathOperator{\cls}{cls} 
\DeclareMathOperator{\Int}{Int} 
 
\begin{document}
\title[Attractors for non-autonomous reaction-diffusion equations]
{Global and cocycle attractors for non-autonomous reaction-diffusion equations. The case of null upper Lyapunov exponent}
\author[T. Caraballo]{Tom\'{a}s Caraballo}
\author[J.A. Langa]{Jos\'{e} A. Langa}
\author[R. Obaya]{Rafael Obaya}
\author[A.M. Sanz]{Ana M. Sanz}
\address[T. Caraballo and J.A. Langa]{Departamento de Ecuaciones Diferenciales y An\'{a}lisis Num\'{e}rico, Universidad de Sevilla, Apdo. de Correos 1160, 41080 Sevilla, Spain} \email{caraball@us.es} \email{langa@us.es}
\address[R. Obaya]{Departamento de Matem\'{a}tica
Aplicada, E. Ingenier\'{\i}as Industriales, Universidad de Valladolid,
47011 Valladolid, Spain, and member of IMUVA, Instituto de Investigaci\'{o}n en
Matem\'{a}ticas, Universidad de Valladolid.}
 \email{rafoba@wmatem.eis.uva.es}
\address[A.M. Sanz]{Departamento de Did\'{a}ctica de las Ciencias Experimentales, Sociales y de la Matem\'{a}tica,
Facultad de Educaci\'{o}n, Universidad de Valladolid, 34004 Palencia, Spain,
and member of IMUVA, Instituto de Investigaci\'{o}n en  Mate\-m\'{a}\-ti\-cas, Universidad de
Valladolid.} \email{anasan@wmatem.eis.uva.es}
\thanks{R. Obaya and A.M. Sanz were partly supported by MINECO/FEDER
under project MTM2015-66330, and the European Commission under project H2020-MSCA-ITN-2014 643073 CRITICS}
\thanks{T. Caraballo and J.A.  Langa were partially supported by  Junta de Andaluc\'{\i}a under Proyecto de Excelencia FQM-1492 and  FEDER Ministerio de Econom\'{\i}a y Competitividad grant MTM2015-63723-P}
\date{}
\begin{abstract}
In this paper we obtain a detailed description of the global and cocycle attractors for the skew-product semiflows induced by the mild solutions of a family of scalar linear-dissipative parabolic problems over a minimal and uniquely ergodic flow.  We consider the case of null upper Lyapunov exponent for the linear part of the problem. Then, two different types of attractors can appear, depending on whether the linear equations have a bounded or an unbounded associated real cocycle. In the first case (e.g.~in periodic equations), the structure of the attractor is simple, whereas in the second case (which occurs in aperiodic equations), the attractor is a pinched set with a complicated structure. We describe situations when the attractor is chaotic in measure in the sense of Li-Yorke. Besides, we obtain a non-autonomous discontinuous pitchfork bifurcation scenario for concave equations, applicable for instance to a linear-dissipative version of the  Chafee-Infante equation.
\end{abstract}
\renewcommand{\subjclassname}{\textup{2010} Mathematics Subject Classification}
\maketitle
\section{Introduction}\label{sec-intro}\noindent
In this paper we investigate the dynamical structure of the global and cocycle attractors of the skew-product semiflow generated by a family of scalar linear-dissipative  reaction-diffusion equations over a minimal and uniquely ergodic flow, with Neumann or Robin boundary conditions. We assume that the terms involved in the equations satisfy standard regularity assumptions which provide the existence, uniqueness, global definition and continuous dependence of mild solutions with respect to initial conditions.
\par
If $P$ denotes the hull of the time-dependent coefficients of a particular equation, then often the flow defined by time-translation on $P$ is minimal and uniquely ergodic, with a unique ergodic measure $\nu$. These are the hypotheses assumed in this
paper, which in particular includes the case of
almost periodic equations.  If $U\subset \R^m$ denotes the spatial domain of the equation, the coefficients of the differential equations are continuous functions from  $P \times \bar U \times \R$ to $\R$ that  can be identified with continuous functions from $P  \times C(\bar U)$ to $C(\bar U)$.  In this formalism the solutions of the linear-dissipative equations generate a continuous global  skew-product semiflow   $\tau$ on $P \times C(\bar U)$.
\par
In this work we analyze the structure of the attractors when $\lambda_P$, the upper Lyapunov exponent of the linear part of the linear-dissipative equations, is null and the flow on $P$ is not periodic. We prove that generically the global attractor $\A= \cup_{p\in P} \{p\}\times A(p)$ is a pinched compact set with ingredients of dynamical complexity like sensitive dependence in relevant subsets of this compact set (see Glasner and Weiss~\cite{glwe}). We establish conditions on the coefficient of the linear equations that provide nontrivial sections $A(p)$ for the elements $p$ in an invariant subset $P_f$ of $P$ with complete measure and prove that, in this case, the restriction of the flow on the global attractor $(\A, \tau)$ is chaotic in the sense of Li-Yorke.  We can understand these results in the framework of non-autonomous bifurcation theory as a discontinuous pitchfork bifurcation of minimal sets.
\par
We next describe the structure and main results of the paper. Section~\ref{sec-preli} contains some basic facts in non-autonomous dynamical systems which will be required in the rest of the paper.
\par
In Section \ref{sec-mild solutions} we review the construction of the skew-product semiflow induced by the mild solutions of a very general family of parabolic partial differential equations (PDEs for short) over a minimal flow $(P,\theta,\R)$ just denoted by $\theta_tp=p{\cdot}t$. We also state a result on comparison of solutions and the strong monotonicity of the semiflow.
\par
Section \ref{sec-linear} is devoted to the study of families of scalar parabolic linear PDEs $\partial y /\partial t  =  \Delta \, y+h(p{\cdot}t,x)\,y$, $t>0$, $x\in U$ for each $p\in P$ ($P$ a minimal and uniquely ergodic flow) with Neumann or Robin boundary conditions. Mild solutions generate a linear skew-product semiflow $\tau_L$  on $P \times C(\bar U)$ which is strongly monotone and hence admits a continuous separation $C(\bar U)= X_1(p) \oplus X_2(p)$, for every $p \in P$, in the terms stated in Pol\'{a}\v{c}ik and Tere\v{s}\v{c}\'{a}k~\cite{pote} and Shen and Yi~\cite{shyi}. The restriction of $\tau_L$ to the principal bundle $\cup_{p\in P}\{p\}\times X_1(p)$ generates a continuous 1-dim linear cocycle $c(t,p)$ whose Lyapunov exponents match the upper Lyapunov exponent $\lambda_P$.  We prove the continuous dependence of  $\lambda_P(h)$ on $h$, the  coefficient in the equations. Consequently, the set $C_0(P \times\bar U)= \{h \in C(P \times \bar U) \mid \lambda_P(h)=0\}$ is closed. We denote by $B(P\times \bar U)$ the subset of $C_0(P \times \bar U)$ formed by the functions $h$ with an associated coboundary cocycle $\ln c(t,p)$, that is, there is a $k \in C(P)$ such that $\ln c(t,p)= k(p{\cdot}t)-k(p)$ for any $p\in P$ and $t\in\R$.
\par
We show, on the one hand, that if the coefficient $h$ is in $B(P\times \bar U)$,  then $P \times \Int C_+(\bar U)$ contains  minimal sets that are copies of the base $P$, all the nontrivial positive solutions are strongly positive,  remain uniformly away from 0 and bounded above and eventually approximate solutions with the same recurrence in time as that of the initial problem; for instance, they are asymptotically almost periodic if the base flow is almost periodic. On the other hand, if the linear coefficient $h$ is in $\mathcal{U}(P \times \bar U)=C_0(P \times \bar U)\setminus B(P \times \bar U)$, then $P \times C_+(\bar U)$ contains pinched compact invariant sets, all the nontrivial positive solutions are strongly positive and for  all the equations given by $p$ in a residual subset of P their modulus oscillates from $0$ to $\infty$ as time goes to $\infty$. In addition, when the base flow on $P$ is aperiodic, we deduce that  $\mathcal{U}(P \times \bar U)$ is a  residual subset of $C_0(P \times \bar U)$ as its complementary set $B(P \times\bar U)$ is a dense subset of first category. The above arguments allow us to show that $\lambda_P(h)$ has a strictly convex variation on $h$, which becomes linear in the trivial case when  $h_2-h_1$, the  difference of the  coefficients involved in the convex combination, only depends on  $p \in P$.
\par
In Section \ref{sec-nonlinear} we study  the behaviour of  the solutions of a family of scalar linear-dissipative  reaction-diffusion equations $\partial y/\partial t  =  \Delta \, y+h(p{\cdot}t,x)\,y+g(p{\cdot}t,x,y)$, $t>0$, $x\in U$, for each $p\in P$, and the dynamical properties of the induced skew-product semiflow $\tau$ on $P \times C(\bar U)$. Standard arguments taken from Caraballo and Han~\cite{caha} or Carvalho et al.~\cite{calaro}  allow us to  deduce the existence of a global attractor $\A= \cup_{p\in P} \{p\}\times A(p)$, and thus $\{A(p)\}_{p \in P} $ defines the cocycle attractor of the continuous skew-product semiflow. In addition, for each $p \in P$  the  invariant family of compact sets $\{A(p{\cdot}t)\}_{t \in \R}$ provides the pullback atractor of the process generated by the solutions of the parabolic equation obtained by evaluation of the coefficients along the trajectory of $p$.
\par
The structure of the global and cocycle attractors in the case that  $\lambda_P$, the upper Lyapunov exponent  of the  linear part of the reaction-diffusion equations, is different from zero has been investigated in Cardoso et al.~\cite{cardoso}. In this work we study the same problem when $\lambda_P=0$ to show that these attractors exhibit a rich dynamics that frequently contains ingredients of high complexity. The global attractor has upper and lower boundaries given by the graphs of two semicontinuous functions $a$ and $b$. For simplicity  we asume that the coefficients of the equation are odd with respect to the dependent variable $y$, which implies that $a=-b$.
\par
More precisely, if $h$, the coefficient of the linear part, is in $B(P \times \bar U)$, then $b$ is continuous and strongly positive and the global atractor is included in the principal bundle, whereas if $h$ is in $\mathcal{U} (P \times \bar U)$, then there is a residual invariant  subset $P_{\rm s} \subset P$  such that $b(p)=0$ for every $p \in P_{\rm s}$ and $P_{\rm f}=P\setminus P_{\rm s}$ is a dense invariant subset of first category with $b(p)\gg 0$  for every $p \in P_{\rm f}$. We prove that $p \in P_{\rm s}$ if and only if $\sup_{t\leq 0} c(t,p)= \infty$  and conversely $p \in P_{\rm f}$ if and only if $\sup_{ t\leq 0} c(t,p) < \infty$. Later we describe precise examples of  functions $h$ such that $\nu (P_{\rm f})=1$ and prove that in this case the restriction of the equations on the section $A(p)$ of the attractor is linear for almost every $p \in P$ and the flow $(\A, \tau)$ is fiber-chaotic in measure in the sense of Li-Yorke.  In consequence, the main results on the structure and properties of the attractors obtained in Caraballo et al.~\cite{caloNonl} for scalar almost periodic linear-dissipative ordinary differential equations (ODEs for short) remain valid for the class of reaction-diffusion models here considered.
\par
Finally,  we introduce a parameter in the equations and analyze the evolution of the structure of the global attractor when the upper Lyapunov exponent of the linear part crosses through zero. Assuming that the nonlinear term is concave and using  results by N\'{u}\~{n}ez et al.~\cite{nuos4} we show that this transition provides a discontinuous bifurcation of attractors and describes a discontinuous pitchfork  bifurcation diagram for the minimal sets. The results in this work offer a dynamical description of the often complicated structure of the global attractor at the bifurcation point.
\section{Basic notions}\label{sec-preli}\noindent
In this section we include some preliminaries about
topological dynamics for non-autonomous dynamical systems.
\par
Let $(P,d)$ be a compact
metric space. A real {\em continuous flow\/} $(P,\theta,\R)$ is
defined by a continuous map $\theta: \R\times P \to  P,\;
(t,p)\mapsto \theta(t,p)=\theta_t(p)=p{\cdot}t$ satisfying
\begin{enumerate}
\renewcommand{\labelenumi}{(\roman{enumi})}
\item $\theta_0=\text{Id},$
\item $\theta_{t+s}=\theta_t\circ\theta_s$ for each $s$, $t\in\R$\,.
\end{enumerate}
The set $\{ \theta_t(p) \mid t\in\R\}$ is called the {\em orbit\/}
of the point $p$. We say that a subset $P_1\subset P$ is {\em
$\theta$-invariant\/} if $\theta_t(P_1)=P_1$ for every $t\in\R$.
The flow $(P,\theta,\R)$ is called {\em minimal\/} if it does not contain properly any other
compact $\theta$-invariant set, or equivalently,  if every
orbit is dense. The flow is {\em distal\/} if the orbits of  any two distinct
points $p_1,\,p_2\in P$  keep at a positive distance,
that is, $\inf_{t\in \R}d(\theta(t,p_1),\theta(t,p_2))>0$; and it is {\em almost periodic\/} if the family of maps $\{\theta_t\}_{t\in \R}:P\to P$ is uniformly equicontinuous. An almost periodic flow is always distal.
\par
A finite regular measure defined on the Borel sets of $P$ is called
a Borel measure on $P$. Given $\mu$ a normalized Borel measure on
$P$, it is {\em $\theta$-invariant\/} if $\mu(\theta_t(P_1))=\mu(P_1)$ for every Borel subset
$P_1\subset P$ and every $t\in \R$. It is {\em ergodic\/}  if, in
addition, $\mu(P_1)=0$ or $\mu(P_1)=1$ for every
$\theta$-invariant subset $P_1\subset P$.
We  denote by $\mathcal{M}(P)$ the set of
all positive and normalized $\theta$-invariant measures on $P$. This set is nonempty by the
Krylov\nbd-Bogoliubov theorem  when $P$ is a
compact metric space. We
say that $(P,\theta,\R)$ is {\em uniquely ergodic\/} if it has a
unique normalized invariant measure, which is then necessarily
ergodic. A minimal and almost periodic flow $(P,\theta,\R)$ is uniquely ergodic.
\par
A standard method to, roughly speaking, get rid of the time variation in a non-autonomous equation and build a non-autonomous dynamical system, is the so-called {\em hull\/} construction. More precisely, a function $f\in C(\R\times\R^m)$ is said to be {\em admissible\/} if for any
compact set $K\subset \R^m$, $f$ is bounded and uniformly continuous
on $\R\times K$. Provided that $f$ is admissible, its {\em hull\/} $P$ is the closure for the compact-open topology of the set of $t$-translates of $f$, $\{ f_t \mid t\in\R\}$ with $f_t(s,x)=f(t+s,x)$
for $s\in \R$ and $x\in\R^m$. The translation map $\R\times P\to P$,
$(t,p)\mapsto p{\cdot}t$ given by $p{\cdot}t(s,x)= p(s+t,x)$ ($s\in \R$ and $x\in\R^m$) defines a
continuous flow on the compact metric space $P$. This flow is minimal as far as the map $f$ has certain recurrent behaviour in time, such as periodicity, almost periodicity, or other weaker properties of recurrence. If the map $f(t,x)$ is uniformly almost periodic (that is, it is admissible and almost periodic in $t$ for any fixed $x$), then the flow on the hull is minimal and almost periodic. It is relevant to note that any minimal and uniquely ergodic flow which is not almost periodic is sensitive with respect to initial conditions (see Glasner and Weiss~\cite{glwe}).
\par
Let $\R_+=\{t\in\R\,|\,t\geq 0\}$. Given a continuous compact flow $(P,\theta,\R)$ and a
complete metric space $(X,\dd)$, a continuous {\em skew-product semiflow\/} $(P\times
X,\tau,\,\R_+)$ on the product space $P\times X$ is determined by a continuous map
\begin{equation*}
 \begin{array}{cccl}
 \tau \colon  &\R_+\times P\times X& \longrightarrow & P\times X \\
& (t,p,x) & \mapsto &(p{\cdot}t,u(t,p,x))
\end{array}
\end{equation*}
 which preserves the flow on $P$, referred to as the {\em base flow\/}.
 The semiflow property means that
\begin{enumerate}
\renewcommand{\labelenumi}{(\roman{enumi})}
\item $\tau_0=\text{Id},$
\item $\tau_{t+s}=\tau_t \circ \tau_s\;$ for all  $\; t$, $s\geq 0\,,$
\end{enumerate}
where again $\tau_t(p,x)=\tau(t,p,x)$ for each $(p,x) \in P\times X$ and $t\in \R_+$.
This leads to the so-called semicocycle property,
\begin{equation*}
 u(t+s,p,x)=u(t,p{\cdot}s,u(s,p,x))\quad\mbox{for $s,t\ge 0$ and $(p,x)\in P\times X$}\,.
\end{equation*}
\par
The set $\{ \tau(t,p,x)\mid t\geq 0\}$ is the {\em semiorbit\/} of
the point $(p,x)$. A subset  $K$ of $P\times X$ is {\em positively
invariant\/} if $\tau_t(K)\subseteq K$
for all $t\geq 0$ and it is $\tau$-{\em invariant\/} if $\tau_t(K)= K$
for all $t\geq 0$.  A compact $\tau$-invariant set $K$ for the
semiflow  is {\em minimal\/} if it does not contain any nonempty
compact $\tau$-invariant set  other than itself.
\par
A compact $\tau$-invariant set $K\subset P\times X$ is called a {\em pinched\/} set if there exists a residual set $P_0\subsetneq P$ such that for every $p\in P_0$ there is a unique element in $K$ with $p$ in the first component, whereas there are more than one if $p\notin P_0$.
\par
The reader can find in  Ellis~\cite{elli}, Sacker and
Sell~\cite{sase}, Shen and Yi~\cite{shyi} and references therein, a
more in-depth survey on topological dynamics.
\par
We now state the definitions of global attractor and cocycle attractor for skew-product semiflows. The books by    Caraballo and Han~\cite{caha},  Carvalho et al.~\cite{calaro} and Kloeden and Rasmussen~\cite{klra}  are  good references for this topic.
\par
We say that the skew-product semiflow $\tau$ has a {\em global attractor\/} if there exists an invariant compact set attracting bounded sets forwards in time; more precisely, if there is a compact set $\A\subset P\times X$ such that $\tau_t(\A) = \A$ for any $t\geq 0$ and $\lim_{t\to\infty} {\rm dist}(\tau_t(\B),\A)=0$ for any bounded set $\B\subset P\times X$, for the semi-Hausdorff distance.
\par
A {\em non-autonomous set\/} is a family $\{A(p)\}_{p\in P}$ of subsets of $X$ indexed by $p\in P$. It is said to be {\em compact\/} provided that $A(p)$ is a compact set in $X$ for every $p\in P$; and it is said to be {\em invariant\/} if for every $p\in P$, $u(t,p,A(p))=A(p{\cdot}t)$ for any $t\geq 0$.  A compact invariant non-autonomous set $\{A(p)\}_{p\in P}$ is called a {\em cocycle attractor\/} for the skew-product semiflow $\tau$ if it pullback attracts all bounded subsets $B\subset X$, that is, for any $p\in P$,
\[
\lim_{t\to\infty} {\rm dist}(u(t,p{\cdot}(-t),B),A(p))=0\,.
\]
It is well-known (see \cite{klra}) that, with $P$ compact, if $\A$ is a global attractor for $\tau$, then $\{A(p)\}_{p\in P}$, with $A(p)=\{x\in X\mid (p,x)\in \A\}$ for each $p\in P$, is a cocycle attractor.
\par
To finish, we include some basic notions on monotone skew-product semiflows. When the state space $X$ is a strongly ordered Banach space, that is, there is a closed convex solid cone of nonnegative vectors $X_+$ with a nonempty interior, then, a (partial) {\em strong order relation\/} on $X$ is
defined by
\begin{equation}\label{order}
\begin{split}
 x\le y \quad &\Longleftrightarrow \quad y-x\in X_+\,;\\
 x< y  \quad &\Longleftrightarrow \quad y-x\in X_+\;\text{ and }\;x\ne y\,;
\\  x\ll y \quad &\Longleftrightarrow \quad y-x\in \Int X_+\,.\qquad\quad\quad~
\end{split}
\end{equation}
In this situation, the skew-product semiflow $\tau$
is {\em monotone\/} if
\begin{equation*}
 u(t,p,x)\le u(t,p,y)\,\quad \text{for\, $t\ge 0$\,,\, $p\in P$ \,and\,
 $x,y\in X$ \,with\, $x\le y$}\,.
\end{equation*}
\par
A Borel map $a:P\to X$ such that $u(t,p,a(p))$ exists for any $t\geq 0$ is said to be
\begin{itemize}
\item[(a)] an {\em equilibrium\/} if $a(p{\cdot}t)=u(t,p,a(p))$ for any $p\in P$ and $t\geq 0$;
\item[(b)] a {\em sub-equilibrium\/} if $a(p{\cdot}t)\leq u(t,p,a(p))$ for any $p\in P$ and $t\geq 0$;
\item[(c)] a {\em super-equilibrium\/} if $a(p{\cdot}t)\geq u(t,p,a(p))$ for any $p\in P$ and $t\geq 0$.
\end{itemize}
A super-equilibrium (resp.~sub-equilibrium) $a:P\to X$ is {\em strong\/} if for some $t_*>0$, $a(p{\cdot}t_*)\gg u(t_*,p,a(p))$  (resp.~$\ll$) for every $p\in P$. The study of semicontinuity properties of these maps and other related issues can be found in Novo et al.~\cite{nono2}.
\section{Skew-product semiflow induced by scalar parabolic PDEs}\label{sec-mild solutions}\noindent
Let us consider a family of scalar parabolic PDEs over a minimal flow $(P,\theta,\R)$, with Neumann or Robin boundary
conditions
\begin{equation}\label{family}
\left\{\begin{array}{l} \des\frac{\partial y}{\partial t}  =
 \Delta \, y+f(p{\cdot}t,x,y)\,,\quad t>0\,,\;\,x\in U, \;\, \text{for each}\; p\in P,\\[.2cm]
By:=\alpha(x)\,y+\des\frac{\partial y}{\partial n} =0\,,\quad  t>0\,,\;\,x\in \partial U,\,
\end{array}\right.
\end{equation}
where $p{\cdot}t$ denotes the flow on $P$; $U$, the spatial domain, is a bounded, open and
connected  subset of $\R^m$ ($m\geq 1$) with a sufficiently smooth boundary
$\partial U$; $\Delta$ is the Laplacian operator on $\R^m$;
$f$ satisfies the following hypothesis:
\begin{itemize}
\item[(H)] $f\colon P\times \bar U\times \R\to\R$ is continuous and  is Lipschitz in $y$ in bounded sets, uniformly for $p\in P$ and $x\in\bar U$, that is, given any $R>0$ there exists an $L_R>0$ such that
\[
|f(p,x,y_2)-f(p,x,y_1)|\leq L_R\,|y_2-y_1|
\]
for any $p\in P$, $x\in\bar U$ and $y_1,\,y_2\in\R$ with $|y_1|,\,|y_2|\leq R$\,;
\end{itemize}
$\partial/\partial n$ denotes the
outward normal
derivative at the  boundary; and
$\alpha:\partial U\to \R$ is a nonnegative  sufficiently regular function.
\par
In order to immerse the initial boundary value  problem (IBV problem for short) associated with the parabolic problem~\eqref{family} into an abstract Cauchy problem (ACP for short), we consider the strongly ordered Banach space $X=C(\bar U)$ of the continuous functions on $\bar U$ with the sup-norm $\|\,{\cdot}\,\|$, and positive cone $X_+=\{z\in X\mid z(x)\geq 0 \;\forall x\in \bar U\}$ with nonempty interior $\Int X_+=\{z\in X\mid z(x)> 0 \;\forall x\in \bar U\}$,  which induces a (partial) strong ordering in $X$ as in~\eqref{order}. Note that $X$ is also a Banach algebra for the usual product $(z_1\,z_2)(x)=z_1(x)\,z_2(x)$ for $z_1,\,z_2\in X$ and  $x\in \bar U$.
\par
Now, following Smith~\cite{smit}, let  $A$ be  the closure of the differential operator
$A_0\colon  D(A_0)\subset X\to X$,  $A_0z=\Delta\,z$, defined on
\[
 D(A_0)=\{z \in C^2(U)\cap C^1(\bar U)\;\mid\;A_0z\in  C(\bar U),\;
 Bz=0 \,\hbox{ on }\partial U\}\,.
\]

The operator $A$ is sectorial and it generates
an analytic compact semigroup of operators $\{T(t)\}_{t\geq 0}$ on
$X$ which is strongly continuous (that is, $A$ is densely defined).
\par
If we define $\tilde f:P\times X\to X$, $(p,z)\mapsto \tilde f(p,z)$, $\tilde f(p,z)(x)=f(p,x,z(x))$,  $x\in \bar U$, the regularity conditions (H) on $f$ are transferred to  $\tilde f$. This leads to the continuity of  $\tilde f$, and Lipschitz continuity with respect to $z$ on any bounded set of $X$ with Lipschitz constant independent of $p$, that is, given any bounded set  $B\subset X$, there exists an $L_B>0$ such that
\[
\|\tilde f(p,z_2)-\tilde f(p,z_1)\|\leq L_B\,\|z_2-z_1\|\quad\text{for any}\,\; p\in P,\; z_1,\,z_2\in B\,.
\]

With the former conditions on $A$ and these conditions on $\tilde f$, when
we consider the ACP given
for each fixed $p\in P$ and $z\in X$ by
\begin{equation}\label{acpnl}
\left\{\begin{array}{l} u'(t)  =
 A\, u(t)+\tilde f(p{\cdot}t,u(t))\,,\quad t>0\,,\\
u(0)=z\,,
\end{array}\right.
\end{equation}
this  problem has a unique  {\it mild solution\/}, that is, there exists a unique continuous map $u(t)=u(t,p,z)$ defined on a maximal interval  $[0,\beta)$ for some $\beta=\beta(p,z)>0$ (possibly $\infty$) which satisfies the integral equation
\begin{equation*}
 u(t)=T(t)\,z +\int_0^t T(t-s)\,
 \tilde f(p{\cdot}s,u(s))\,ds\,,\quad t\in [0,\beta)\,.
\end{equation*}
(For instance, see Travis and Webb~\cite{trwe} or Hino et al.~\cite{hnms}.)
Mild solutions allow us to locally define a continuous
skew-product semiflow
\begin{equation*}
\begin{array}{cccl}
 \tau: &\mathcal{U} \subseteq\R_+\times P\times X& \longrightarrow & \hspace{0.3cm}P\times X\\
 & (t,p,z) & \mapsto
 &(p{\cdot}t,u(t,p,z))\,,
\end{array}
\end{equation*}
for an appropriate open set $\mathcal{U}$. Besides, if a solution $u(t,p,z)$ remains bounded, then it is defined on the whole positive real line and the semiorbit of $(p,z)$ is relatively compact (see Proposition~2.4 in~\cite{trwe}, where the compactness of the operators $T(t)$ for $t>0$ is crucial).
\par
Note that the linear family
\begin{equation}\label{pdefamily}
\left\{\begin{array}{l} \des\frac{\partial y}{\partial t}  =
 \Delta \, y+h(p{\cdot}t,x)\,y\,,\quad t>0\,,\;\,x\in U, \;\, \text{for each}\; p\in P,\\[.2cm]
By:=\alpha(x)\,y+\des\frac{\partial y}{\partial n} =0\,,\quad  t>0\,,\;\,x\in \partial U,\,
\end{array}\right.
\end{equation}
with $h:P\times\bar U\to\R$ a continuous map, is included in the general setting of \eqref{family}. In this case,  $\tilde h:P\to X$, $p\mapsto \tilde h(p)$, $\tilde h(p)(x)=h(p,x)$,  $x\in \bar U$ is continuous and bounded. In the associated linear ACP  given  for each $p\in P$ and $z\in X$ by
\begin{equation}\label{acplineal}
\left\{\begin{array}{l} v'(t)  =
 A\, v(t)+\tilde h(p{\cdot}t)\,v(t)\,,\quad t>0\,,\\
v(0)=z\,,
\end{array}\right.
\end{equation}
there appears the term $\tilde f(p,z)=\tilde h(p)\,z$, for $p\in P$, $z\in X$  which  is globally Lipschitz continuous with respect to $z$, uniformly for $p\in P$. This implies that the mild solutions  $v(t)=v(t,p,z)$, which in this linear case are solutions of  the integral equations
\begin{equation*}
 v(t)=T(t)\,z +\int_0^t T(t-s)\,
 \tilde h(p{\cdot}s)\,v(s)\,ds\,,\quad t\geq 0\,,
\end{equation*}
allow us to define a globally defined  continuous linear
skew-product semiflow
\begin{equation*}
\begin{array}{cccl}
 \tau_L: & \R_+\times P\times X& \longrightarrow & \hspace{0.3cm}P\times X\\
 & (t,p,z) & \mapsto
 &(p{\cdot}t,\phi(t,p)\,z)\,,
\end{array}
\end{equation*}
where $\phi(t,p)\,z=v(t,p,z)$. In particular $\phi(t,p)$ are bounded operators on $X$ which are compact for $t>0$  and satisfy the linear semicocycle property $\phi(t+s,p)=\phi(t,p{\cdot}s)\,\phi(s,p)$, $p\in P$, $t,s\geq 0$. As before, bounded trajectories are relatively compact.
\par
Under additional regularity conditions in the nonlinear term $f(p{\cdot}t,x,y)$, such as a Lipschitz condition with respect to $t$ and H\"{o}lder-continuity  with respect to $x$, mild solutions are known to generate classical solutions; namely, $y(t,x)=u(t,p,z)(x)$, $t\in [0,\beta(p,z))$, $x\in \bar U$ is a classical solution of the IBV problem given by~\eqref{family}  for $p\in P$ with initial condition at time $t=0$, $y(0,x)=z(x)$, $x\in \bar U$,  meaning that  the corresponding partial derivatives exist, are continuous
and satisfy the corresponding equation in~\eqref{family} as well as the boundary conditions (see Smith~\cite{smit} and
Friedman~\cite{frie}).
\par
We next state a result of comparison of solutions which will be used through the paper, and the strong monotonicity of the semiflow.
\begin{teor}\label{teor-comparacion}
Let $f_1$ and $f_2$ satisfy hypothesis $\rm{(H)}$ and be such that $f_1\leq f_2$. For each $p\in P$ and $z\in X$, denote by $u_1(t,p,z)$ and $u_2(t,p,z)$ the mild solutions of the associated ACPs~\eqref{acpnl}, respectively. Then, $u_1(t,p,z)\leq u_2(t,p,z)$ for any $t\geq 0$ where both solutions are defined.
\end{teor}
\begin{proof}
Let us fix a $p\in P$ and a $z\in X$ and let $t_0>0$ be such that both $u_1(t_0,p,z)$ and $u_2(t_0,p,z)$ exist. The idea is to approximate the equations given by $f_1$ and $f_2$ by a sequence of equations to which the standard comparison of solutions result applies, and whose solutions approximate the mild solutions of the initial problems. \par
Let $R=\sup\{\|u_1(t,p,z)\|,\,\|u_2(t,p,z)\|\mid t\in [0,t_0]\}<\infty$. First, we apply Tietze's extension theorem to the continuous  map
\[
g_i:[0,t_0]\times\bar U\times [-R,R]\to \R\,,\quad  (t,x,y)\mapsto f_i(p{\cdot}t,x,y)
\]
for $i=1,2$. Thus there exist continuous maps $F_i:\R\times\R^m\times \R\to \R$ ($i=1,2$) with compact support such that the restriction  $F_i |_{[0,t_0]\times\bar U\times [-R,R]}\equiv g_i$ and $\|F_i\| = \|g_i\|_{[0,t_0]\times\bar U\times [-R,R]}$. Now, for $i=1,2$, we apply to $F_i$ the regularization process used in the construction of solutions of the heat equation, using the convolution with the so-called {\em Gauss kernel\/}; namely, the maps defined on $\R\times \R^m\times \R$,
\[
F_{i,n}(t,x,y)=\left(\frac{n}{4\pi}\right)^{\!\frac{m+2}{2}}\int_{\R\times\R^m\times\R} e^{\frac{-n\,\|(t,x,y)-(\wit t,\wit x,\wit y)\|^2}{4}}F_i(\wit t,\wit x,\wit y)\,d\wit t\,d\wit x\,d\wit y\,,\quad \;n\geq 1
\]
satisfy:
\begin{itemize}
  \item[(i)] $F_{i,n}(t,x,y)$ is of class $C^\infty$ with respect to $t$, $x$ and $y$;
  \item[(ii)] $\des\lim_{n\to\infty} F_{i,n}(t,x,y)=F_i(t,x,y)$ uniformly;
  \item[(iii)] $|F_{i,n}(t,x,y)|\leq \|F_i\|$ for any $(t,x,y)\in \R\times\R^m\times\R$.
\end{itemize}
\par
At this point, for $i=1,2$ and for each $n\geq 1$ we denote by $u_{i,n}(t,p,z)$ the mild solution of the  ACP for $p$ and $z$ given by
\begin{equation*}
\left\{\begin{array}{l} u'(t)  =
 A\, u(t)+\widetilde F_{i,n}(t,u(t))\,,\quad t>0\,,\\
u(0)=z\,,
\end{array}\right.
\end{equation*}
where $\widetilde F_{i,n}(t,u)(x)=F_{i,n}(t,x,u(x))$,  $x\in \bar U$. Thanks to (iii),  $u_{i,n}(t,p,z)$ is defined on $[0,t_0]$ for $n\geq 1$, and we affirm that $u_{i,n}(t,p,z)\to u_i(t,p,z)$ uniformly for $t\in [0,t_0]$ as $n\to\infty$. To see it,  follow the argumentation in the proof of Proposition~3.2 in Novo et al.~\cite{nonuobsa} (inspired in the proof of Proposition~2.4 in Travis and Webb~\cite{trwe}).
\par
To finish, note that we can also assume that $F_{1,n}\leq F_1 \leq F_2\leq F_{2,n}$ for any $n\geq 1$ and besides, with the regularity conditions we have on $F_{i,n}$, the mild solutions of the ACPs give rise to classical solutions of the associated IBV problems. Therefore,  the standard result of comparison of solutions says that $u_{1,n}(t,p,z)\leq u_{2,n}(t,p,z)$ for $t\in [0,t_0]$ for every $n\geq 1$. Therefore, taking limits, we finally obtain that $u_1(t,p,z)\leq u_2(t,p,z)$ for $t\in [0,t_0]$, as desired.
\end{proof}
\begin{prop}\label{prop-strong monot lineal}
Consider the linear problem \eqref{pdefamily} with $h:P\times\bar U\to\R$ continuous. Then, the induced linear skew-product semiflow $\tau_L$ is strongly monotone, that is, for any $p\in P$, $\phi(t,p)\,z\gg 0$  whenever  $z>0$, for any $t>0$.
\end{prop}
\begin{proof}
We just consider a regular map $h_1:P\times\bar U\to\R$ with $h_1\leq h$. The  linear skew-product semiflow associated to the regular map $h_1$ is well-known to be strongly monotone, so that the result follows by comparison applying Theorem~\ref{teor-comparacion}.
\end{proof}
To finish this section, in the nonlinear case we deduce the strong monotonicity of the induced skew-product semiflow $\tau$ by linearizing, provided that the nonlinear term is of class $C^1$ in the $y$ variable.
\begin{prop}\label{prop-strong monot}
Consider the nonlinear problem \eqref{family} with $f:P\times\bar U\times \R\to\R$ continuous and  of class $C^1$ in the $y$ variable. Then, the induced skew-product semiflow   is  strongly monotone, that is, for any $p\in P$, $u(t,p,z_2)\gg u(t,p,z_1)$  whenever  $z_2>z_1$, for any $t>0$ where both terms are defined.
\end{prop}
\begin{proof}
First of all,  we define the map $\frac{\widetilde {\partial f}}{\partial y}:P\times X \to X$, $(p,z)\mapsto \frac{\widetilde {\partial f}}{\partial y}(p,z)$, $\frac{\widetilde {\partial f}}{\partial y}(p,z)(x)=\frac{\partial f}{\partial y}(p,x,z(x))$, $x\in\bar U$, which is continuous. Then, given a pair $(p,z)\in P\times X$, we consider the associated variational ACP along the trajectory of $(p,z)$ with initial value $z_0\in X$, for $t>0$ as long as $\tau(t,p,z)$ exists:
\begin{equation*}
\left\{\begin{array}{l} v'(t)  =
 A\, v(t)+\des\frac{\widetilde {\partial f}}{\partial y}(\tau(t,p,z))\,v(t)\,,\\
v(0)=z_0\,.
\end{array}\right.
\end{equation*}
Denoting by $v(t,p,z,z_0)$ the mild solution to this problem, we follow the argumentation in the proof of Theorem~3.5 in Novo et al.~\cite{nonuobsa} to see that $D_zu(t,p,z)\,z_0$ exists and $D_zu(t,p,z)\,z_0=v(t,p,z,z_0)$. Besides, the map $P\times X\to \mathcal{L}(X)$, $(p,z)\mapsto D_zu(t,p,z)$ is continuous for any $t>0$ in the interval of definition of $u(t,p,z)$.
\par
To finish the proof, given a $T>0$ such that $u(t,p,z_1)$ and $u(t,p,z_2)$ are defined on $[0,T]$, we can assume without loss of generality that also
$u(t,p,\lambda\,z_2+(1-\lambda)\,z_1)$ is defined on $[0,T]$ for any $\lambda\in(0,1)$, and then just write for $z_1< z_2$,
\begin{equation*}
 u(t,p,z_2)-u(t,p,z_1)=\int_0^1 D_z u(t,p,\lambda\,z_2+(1-\lambda)\,z_1)(z_2-z_1)\,d\lambda\,.
\end{equation*}
\par
By applying Proposition~\ref{prop-strong monot lineal} to the mild solutions of the variational linear ACPs, we get the nonnegativity of the integrand. Since $z_1<z_2$,  at $\lambda=0$ (for instance) apply the strong monotonicity, so that $D_z u(t,p,z_1)(z_2-z_1)\gg 0$, and this, together with the continuity of the integrand, is enough to conclude the proof.
\end{proof}
\section{Scalar linear parabolic PDEs with null upper Lyapunov exponent}\label{sec-linear}\noindent
In this section we concentrate on the linear case. Let us consider a family \eqref{pdefamily} of scalar linear parabolic PDEs over a minimal flow $(P,\theta,\R)$, with Neumann or Robin boundary
conditions:
\begin{equation*}
\left\{\begin{array}{l} \des\frac{\partial y}{\partial t}  =
 \Delta \, y+h(p{\cdot}t,x)\,y\,,\quad t>0\,,\;\,x\in U, \;\, \text{for each}\; p\in P,\\[.2cm]
By:=\alpha(x)\,y+\des\frac{\partial y}{\partial n} =0\,,\quad  t>0\,,\;\,x\in \partial U,\,
\end{array}\right.
\end{equation*}
with $h\in C(P\times \bar U)$, the Banach space of the continuous real maps defined on $P\times \bar U$. We keep  the notation introduced in the previous section; in particular, $\tau_L$ is the globally defined linear skew-product semiflow given by the mild solutions of the associated ACPs, determined by the compact (for $t>0$) linear operators $\phi(t,p)\in \mathcal{L}(X)$. Recall also that $\tau_L$ is strongly monotone: see Proposition~\ref{prop-strong monot lineal}.
\par
The {\em Sacker-Sell spectrum\/} (or  {\em continuous spectrum\/}; see Sacker and Sell~\cite{sase94}) of $\tau_L$ is the set
\[
\Sigma=\{\lambda\in \R\mid \tau_L^\lambda \text{ has no exponential dichotomy} \},
\]
where $\tau_L^\lambda$ denotes the linear skew-product semiflow  $\tau_L^\lambda(t,p,z)=(p{\cdot}t,e^{-\lambda t}\phi(t,p)\,z)$  on $P\times X$. The {\em upper Lyapunov exponent\/} of $\tau_L$ is defined as $\lambda_P=\sup_{p\in P} \lambda(p)$, where $\lambda(p)$ is the  Lyapunov exponent given by
\begin{equation}\label{lyap exp}
\lambda(p)=\limsup_{t\to \infty} \frac{\ln\|\phi(t,p)\|}{t}=\limsup_{t\to \infty} \frac{\ln\|\phi(t,p)\,z\|}{t}\,
\end{equation}
for any $z\gg 0$; since for a given $z\gg 0$ there exists an $l=l(z)>0$ such that
$\|\phi(t,p)\|\leq l\,\|\phi(t,p)\,z\|$ for any $t>0$ and
$p\in P$.
It is well-known that $\lambda_P=\sup \Sigma <\infty$ (see Shen and Yi~\cite{shyi} and Chow and Leiva~\cite{chle}
for further details).
\par
To emphasize the dependance of $\lambda_P$ on the coefficient $h$, we will write $\lambda_P(h)$. In particular, for $h=0$, the problem is autonomous and the solution semiflow is given by the semigroup  $\{T(t)\}_{t\geq 0}$. Since $\|T(t)\|\leq 1$ for any $t\geq 0$ (see Smith~\cite{smit}), it follows that $\lambda_P(0)\leq 0$. One can arrive at this same conclusion by considering the strongly positive solution of problem~\eqref{pdefamily} with $h=0$ given by $y(t,x)=e^{-\gamma_0 t}e_0(x)$, where $\gamma_0\geq 0$ is the  first eigenvalue and $e_0\in X$, with $e_0\gg 0$ and $\|e_0\|=1$ is the associated eigenfunction, of the  boundary value  problem
\begin{equation}\label{bvp}
\left\{\begin{array}{l}
 \Delta \, u +\lambda\,u = 0\,,\quad x\in U,\\
Bu:=\alpha(x)\,u+\des\frac{\partial u}{\partial n} =0\,,\quad x\in \partial U.
\end{array}\right.
\end{equation}
More precisely, it turns out that $\lambda_P(0)=-\gamma_0\leq 0$.
\par
As proved  by Pol\'{a}\v{c}ik and Tere\v{s}\v{c}\'{a}k~\cite{pote} in the discrete case, and then extended by Shen and Yi~\cite{shyi} to the continuous case, the operators $\phi(t,p)$ being compact and strongly positive make the  linear skew-product semiflow $\tau_L$ admit a continuous separation. This means that there are two families of subspaces $\{X_1(p)\}_{p\in P}$ and $\{X_2(p)\}_{p\in P}$ of $X$ which satisfy:
\begin{itemize}
\item[(1)] $X=X_1(p)\oplus X_2(p)$  and $X_1(p)$, $X_2(p)$ vary
    continuously in $P$;
 \item[(2)] $X_1(p)=\langle e(p)\rangle$, with $e(p)\gg 0$ and
     $\|e(p)\|=1$ for any $p\in P$;
\item[(3)] $X_2(p)\cap X_+=\{0\}$ for any $p\in P$;
\item[(4)] for any $t>0$,  $p\in P$,
\begin{align*}
\phi(t,p)X_1(p)&= X_1(p{\cdot}t)\,,\\
\phi(t,p)X_2(p)&\subset X_2(p{\cdot}t)\,;
\end{align*}
\item[(5)] there are $M>0$, $\delta>0$ such that for any $p\in P$, $z\in
    X_2(p)$ with $\|z\|=1$ and $t>0$,
\begin{equation*}
\|\phi(t,p)\,z\|\leq M \,e^{-\delta t}\|\phi(t,p)\,e(p)\|\,.
\end{equation*}
\end{itemize}
\par
In this situation, the 1-dim invariant subbundle
\begin{equation*}
\displaystyle\bigcup_{p\in P} \{p\} \times X_1(p)\,
\end{equation*}
is called the {\em principal bundle\/} and the Sacker-Sell spectrum of the restriction of $\tau_L$ to this invariant subbundle is called the {\em principal spectrum\/} of $\tau_L$, and is denoted by  $\Sigma_{\text{pr}}(\tau_L)$ (see Mierczy{\'n}ski and Shen~\cite{mish}). It is well-known that $\Sigma_{\text{pr}}(\tau_L)$ is a possibly degenerate compact interval of the real line. Actually, if $c(t,p)$ is the real linear semicocycle associated with the continuous separation, that is, if for any $t\geq 0$ and $p\in P$, $c(t,p)$ is the positive number such that
\begin{equation}\label{c}
\phi(t,p)\,e(p)=c(t,p)\,e(p{\cdot}t)\,,
\end{equation}
then the Lyapunov exponents~\eqref{lyap exp} can be calculated by $\lambda(p)=\des\limsup_{t\to\infty}\frac{\ln c(t,p)}{t}$ for each $p\in P$ and besides $\Sigma_{\text{pr}}(L)=[\alpha_P,\lambda_P]$ with $\alpha_P\leq\lambda_P$, and
there are two ergodic measures $\mu_1$, $\mu_2\in\mathcal{M}(P)$ such that
\begin{equation}\label{rep int exponente}
\alpha_{P}=\int_P \ln c(1,p)\,d\mu_1\,\quad \text{and}\quad
\lambda_{P}=\int_P \ln c(1,p)\,d\mu_2\,.
\end{equation}
The reader is referred to Novo et al.~\cite{noos7} for all the details in an abstract setting.
\par
As a consequence, when the flow on $P$ is uniquely ergodic, the principal spectrum is a singleton determined by the upper Lyapunov exponent: $\Sigma_{\text{pr}}(\tau_L)=\{\lambda_P\}$. Furthermore, in the uniquely ergodic setting an application of Birkhoff's ergodic theorem  permits to conclude that $\lambda_P=\lambda(p)$ for any $p\in P$ and besides, the superior limit in the definition of $\lambda(p)$ is an existing limit.
\par
Note that the linear semicocycle $c(t,p)$ can be extended to a linear cocycle just by taking $c(-t,p)=1/c(t,p{\cdot}(-t))$ for any $t>0$ and $p\in P$. Since this 1-dim linear cocycle is going to be a fundamental tool in this section, we give a definition.
\begin{defi}\label{defi-cociclo c}
For each $h\in C(P\times\bar U)$,  $c(t,p)$ ($t\in\R$, $p\in P$) is the 1-dim linear cocycle driving the dynamics of $\tau_L$ when restricted to the principal bundle determined by the continuous separation (see~\eqref{c}).
\end{defi}
The kind of results that we are going to present in the linear case are in line with those in Caraballo et al.~\cite{caloNonl} given for  families of scalar linear ODEs $x'=h(p{\cdot}t)\,x$, $p\in P$, with $P$ a minimal and almost periodic flow and with null upper Lyapunov exponent $\lambda_P=\lambda_P(h)=0$. A significant difference is that in the case of  scalar ODEs $\lambda_P(h)$ is a linear map with respect to $h$,
\[
\lambda_P(h)=\int_P h\,d\nu\,,
\]
for the Haar measure $\nu$ in $P$, whereas in the present case of scalar parabolic PDEs we  show that the dependance of the upper Lyapunov exponent on $h$ is continuous and convex but not linear any more: just note that $\lambda_P(0)=-\gamma_0<0$ for Robin boundary conditions (also, see Theorem~\ref{prop-convex estricta}).
\par
From now on, we assume that the flow on $P$ is minimal and uniquely ergodic and  $\nu$ denotes the unique ergodic measure.
\begin{prop}\label{prop-continua}
The map $\lambda_P:C(P\times\bar U)\to \R$, $h\mapsto \lambda_P(h)$ is continuous. As a consequence,
\begin{equation*}
C_0(P\times\bar U)=\{h\in C(P\times\bar U)\mid \lambda_P(h)=0\}
\end{equation*}
is a closed complete set in $C(P\times\bar U)$.
\end{prop}
\begin{proof}
Let $h\in C(P\times\bar U)$ and let $(h_n)_n\subset C(P\times\bar U)$ be such that $h_n\to h$ as $n\to\infty$. Then, in particular, fixed an $\varepsilon>0$ there exists an $n_0$ such that $h-\varepsilon\leq h_n\leq h+\varepsilon$ for any $n\geq n_0$. Applying Theorem~\ref{teor-comparacion} we deduce that $\lambda_P(h-\varepsilon)\leq  \lambda_P(h_n)\leq \lambda_P(h+\varepsilon)$ for any $n\geq n_0$. Now, since the linear cocycle for $h\pm\varepsilon$ is just given by $\exp(\pm\varepsilon\,t)\,\phi(t,p)$, it is straightforward that $\lambda_P(h\pm\varepsilon)=\lambda_P(h)\pm\varepsilon$, so that $\lambda_P(h_n)\to \lambda_P(h)$ as $n\to\infty$. The proof is finished.
\end{proof}
Let us now deal with the convexity of $\lambda_P(h)$.
\begin{prop}\label{prop-convex}
For any $h_1,\,h_2\in C(P\times\bar U)$ and any $0\leq r\leq 1$,
\[
 \lambda_P(r h_1+(1-r)h_2)\leq r\lambda_P(h_1)+(1-r)\lambda_P(h_2)\,.
\]
\end{prop}
\begin{proof}
First, let us assume that $h_1$ and $h_2$ are regular enough so that the mild solutions of the associated IBV problems  become classical solutions. For a fixed $p\in P$, and any fixed $z_0\in X$, $z_0\gg 0$, on the one hand, let $y_1(t,x)$ and $y_2(t,x)$ denote respectively the solution of the IBV problem for $i=1,2$:
\begin{equation*}
\left\{\begin{array}{l} \des\frac{\partial y}{\partial t}  =
 \Delta \, y+h_i(p{\cdot}t,x)\,y\,,\quad t>0\,,\;\, x\in  U,\,\\[.2cm]
By:=\alpha(x)\,y+\des\frac{\partial y}{\partial n} =0\,,\quad t>0\,,\;\, x\in \partial U,\,\\[.2cm]
y(0,x)=z_0(x)\,,\quad  x\in  \bar U;\,
\end{array}\right.
\end{equation*}
and let $\phi_1(t,p)$ and $\phi_2(t,p)$ be the associated linear cocycles, so that $y_i(t,x)=(\phi_i(t,p)\,z_0)(x)$, $t\geq 0,\,x\in \bar U$, for $i=1,2$.
On the other hand, let $y(t,x)$ be the solution of the IBV problem
\begin{equation*}
\left\{\begin{array}{l} \des\frac{\partial y}{\partial t}  =
 \Delta \, y+(r h_1(p{\cdot}t,x)+(1-r) h_2(p{\cdot}t,x))\,y\,,\quad t>0\,,\;\, x\in  U,\,\\[.2cm]
By:=\alpha(x)\,y+\des\frac{\partial y}{\partial n} =0\,,\quad t>0\,,\;\, x\in \partial U,\,\\[.2cm]
y(0,x)=z_0(x)\,,\quad  x\in  \bar U;\,
\end{array}\right.
\end{equation*}
with associated linear cocycle $\Phi(t,p)$, so that $y(t,x)=(\Phi(t,p)\,z_0)(x)$, $t\geq 0,\,x\in \bar U$. By the strong monotonicity of these problems, $y_1(t,x)$, $y_2(t,x)$, $y(t,x)>0$ for any $t\geq 0$, $x\in \bar U$, and we can consider $z(t,x)=\exp(r\ln y_1(t,x)+(1-r)\ln y_2(t,x))$.
\par
We do some routine calculations for $z(t,x)$:
\begin{align*}
\des\frac{\partial z}{\partial t}  = & z\left(\frac{r}{y_1}\,\frac{\partial y_1}{\partial t}+\frac{1-r}{y_2}\,\frac{\partial y_2}{\partial t} \right)\\ = & z\left(\frac{r}{y_1}\,\Delta\,y_1+\frac{1-r}{y_2}\,\Delta\,y_2 +r h_1(p{\cdot}t,x)+(1-r) h_2(p{\cdot}t,x) \right);
 \\ \des\frac{\partial z}{\partial x_i}  = & z\left(\frac{r}{y_1}\,\frac{\partial y_1}{\partial x_i}+\frac{1-r}{y_2}\,\frac{\partial y_2}{\partial x_i} \right)\Rightarrow \nabla z = z\left(\frac{r}{y_1}\,\nabla y_1 + \frac{1-r}{y_2}\,\nabla y_2  \right);\\
 \Delta\,z =&  z  \sum_{i=1}^m \left( \left( \frac{r}{y_1}\,\frac{\partial y_1}{\partial x_i}+\frac{1-r}{y_2}\,\frac{\partial y_2}{\partial x_i}\right)^{\!2} - \frac{r}{y_1^2}\,\left(\frac{\partial y_1}{\partial x_i}\right)^{\!2}- \frac{1-r}{y_2^2}\,\left(\frac{\partial y_2}{\partial x_i}\right)^{\!2} \right)\\ &  +  z\left( \frac{r}{y_1}\,\Delta\, y_1+\frac{1-r}{y_2}\,\Delta \,y_2 \right) \leq z\left( \frac{r}{y_1}\,\Delta\, y_1+\frac{1-r}{y_2}\,\Delta \,y_2 \right),
 \end{align*}
where the convexity of the map $\R\to\R$, $s\mapsto s^2$ has been applied in the inequality. Therefore, $z(t,x)$ is a solution of the problem
\begin{equation*}
\left\{\begin{array}{l} \des\frac{\partial z}{\partial t}  \geq
 \Delta \, z+(r h_1(p{\cdot}t,x)+(1-r) h_2(p{\cdot}t,x))\,z\,,\quad t>0\,,\;\, x\in  U,\,\\[.2cm]
Bz:=\alpha(x)\,z+\des\frac{\partial z}{\partial n} =0\,,\quad t>0\,,\;\, x\in \partial U,\,\\[.2cm]
z(0,x)=z_0(x)\,,\quad  x\in  \bar U.
\end{array}\right.
\end{equation*}
\par
Then, a standard argument of comparison of solutions (see Smith~\cite{smit}) says that $z(t,x)\geq y(t,x)$, that is,  $y_1(t,x)^ry_2(t,x)^{1-r}\geq y(t,x)$ for $t\geq 0$, $x\in  \bar U$. In other words, we have proved in $X$ that  $(\phi_1(t,p)\,z_0)^r(\phi_2(t,p)\,z_0)^{1-r}\geq \Phi(t,p)\,z_0$. Applying monotonicity of the norm and the fact that $X$ is a Banach algebra,
\[
\|\phi_1(t,p)\,z_0\|^r\|\phi_2(t,p)\,z_0\|^{1-r}\geq \|(\phi_1(t,p)\,z_0)^r(\phi_2(t,p)\,z_0)^{1-r}\|\geq \|\Phi(t,p)\,z_0\|\,,
\]
and taking logarithm,
\begin{equation*}
r\ln \|\phi_1(t,p)\,z_0\|+ (1-r)\ln \|\phi_2(t,p)\,z_0\|\geq \ln \|\Phi(t,p)\,z_0\|\,,\quad t\geq 0\,.
\end{equation*}
\par
As it has been remarked before, in the uniquely ergodic case the upper Lyapunov exponent equals the value of any of the Lyapunov exponents, and in particular that of $p$, so that having~\eqref{lyap exp} in mind,  it suffices to divide by $t$ and take limits as $t\to\infty$ to get the convexity relation.
\par
To finish the proof, consider any $h_1,\,h_2\in C(P\times\bar U)$. Using a result by Schwartzman~\cite{schw} we can approximate these maps by respective sequences $(h_{1,n})_n,\,(h_{2,n})_n$ of sufficiently regular maps. More precisely, maps of class $C^1$ in $U$ and of class $C^{1\!}$ along the orbits in $P$, that is, for any $p\in P$ and $x\in\bar U$ the maps $h_{i,n}(p{\cdot}t,x)$ are continuously differentiable in $t\in\R$ ($i=1,2$, $n\geq 1$). Since the convexity relation applies to the pairs $h_{1,n},\,h_{2,n}$ for any $n\geq 1$, with the continuity result in  Proposition~\ref{prop-continua} we are done.
\end{proof}
As a corollary, since $\lambda_P(0)\leq 0$, we get the superlinear character of $\lambda_P$, that is,
$\lambda_P(r h)\leq r\lambda_P(h)$ for any $h\in C(P\times\bar U)$ and any $0\leq r\leq 1$.
\par
Once we have studied some basic properties of the map $\lambda_P(h)$, our aim is to give a description of the dynamics of the linear semiflow $\tau_L$ when $\lambda_P(h)=0$, depending on the map $h$. As it was also done in Caraballo et al.~\cite{caloNonl}, from now on we assume that the minimal and uniquely ergodic flow on $P$ is not periodic. In $C(P)$, the space of continuous functions on $P$, we consider the Banach  space $C_0(P)=\left\{a\in C(P)\mid \int_P a\,d\nu=0\right\}$,  its  vector subspace
\[
B(P)=\left\{ a\in C_0(P) \;\Big|\; \sup_{t\in\R} \left| \int_0^t a(p{\cdot}s)\,ds\right|<\infty \;\text{ for any}\;p\in P \right\}
\]
of the continuous functions on $P$ with zero mean and bounded primitive, and its complement $\mathcal{U}(P)=C_0(P)\setminus B(P)$ of the continuous functions on $P$ with zero mean and unbounded primitive. As a consequence of Lemma 5.1 in Campos et al.~\cite{caot}, $B(P)$ is a dense set of first category in $C_0(P)$ and thus $\mathcal{U}(P)$ is a residual set (see Gottschalk and Hedlund~\cite{gohe} and Johnson~\cite{john} for the result in the almost periodic and aperiodic case).
\par
Now, in the complete metric space $C_0(P\times\bar U)=\{h\in C(P\times\bar U)\mid \lambda_P(h)=0\}$ 
we introduce the sets
\begin{align*}
B(P\times \bar U) &= \{ h\in C_0(P\times\bar U)\mid \sup_{t\in\R} | \ln c(t,p)|<\infty \;\text{ for any}\;p\in P \}\,\; \text{and} \\
\mathcal{U}(P\times \bar U) &=C_0(P\times\bar U)\setminus B(P\times \bar U) \,,
\end{align*}
for the associated 1-dim linear cocycle $c(t,p)$  given in Definition~\ref{defi-cociclo c}.
Note that the condition determining $B(P\times \bar U)$ is equivalent to saying that for any $p\in P$ the linear positive cocycle $c(t,p)$ is both bounded away from $0$ and bounded above.
\par
Next, we state without proof two technical results given for general positive 1-dim linear cocycles $c(t,p)$, which are in correspondance with two classical results for maps in $C_0(P)$. The first one is the adaptation of Proposition~12 in~\cite{caloNonl} (proved in~\cite{gohe}), whereas the second one has the spirit of the oscillation result stated in Theorem~13 in~\cite{caloNonl} (proved in~\cite{john}). In fact, the proofs can be  adapted respectively from the proofs of Proposition~A.1 and Theorem~A.2 in Jorba et al.~\cite{jnot}.
\begin{prop}\label{prop-gott}
Let $c(t,p)$ be a continuous positive 1-dim linear cocycle. Then, the following conditions are equivalent:
\begin{itemize}
\item[(i)] There exists a function $k\in C(P)$ such that
\[
k(p{\cdot}t)-k(p)=\ln c(t,p) \;\;\text{for all}\;\; p\in P,\;t\in\R\,.
\]
\item[(ii)] For any $p\in P$, $\sup_{t\in\R} | \ln c(t,p)|<\infty $.
\item[(iii)] There exists a  $p_0\in P$ such that  $\sup_{t\in\R} | \ln c(t,p_0)|<\infty $.
\item[(iv)] There exists a  $p_0\in P$ such that
\[
\text{either}\;\; \sup_{t\geq 0} | \ln c(t,p_0)|<\infty \quad\text{or}\;\; \sup_{t\leq 0} | \ln c(t,p_0)|<\infty\,.
\]
\end{itemize}
\end{prop}
\begin{teor}\label{teor-johnson}
Let $c(t,p)$ be a continuous positive 1-dim linear cocycle and assume that it does not satisfy the conditions in Proposition~$\ref{prop-gott}$, and the associated real linear skew-product flow $\R\times P\times\R\to P\times\R$, $(t,p,y)\mapsto (p{\cdot}t,c(t,p)\,y)$ does not have an exponential dichotomy. Then,  there exists an invariant and residual set $P_{\rm{o}}\subset P$ such that for any $p\in P_{\rm{o}}$ there exist sequences  (depending on $p$) $(t_n^i)_n$, $i=1,2,3,4$ with  $t_n^i\uparrow \infty$ for $i=1,2$ and $t_n^i\downarrow -\infty$ for $i=3,4$  such that
\[
\lim_{n\to\infty}c(t_n^i,p) =0 \;\;\text{for}\;\;i=1,3 \quad\text{and}\quad
\lim_{n\to\infty} c(t_n^i, p)=\infty\;\;\text{for}\;\;i=2,4\,.
\]
\end{teor}
We will sometimes refer to $P_{\rm{o}}$ as the oscillation set of $c(t,p)$.
\begin{nota}\label{remark-DE}
Note that if $h\in C_0(P\times\bar U)$, then $\Sigma_{\rm{pr}}(\tau_L)=\{0\}$, that is, $\tau_L$ restricted to the principal bundle does not have an exponential dichotomy. In other words, the associated real linear skew-product flow $\R\times P\times\R\to P\times\R$, $(t,p,y)\mapsto (p{\cdot}t,c(t,p)\,y)$ does not have an exponential dichotomy. This means that given any $h\in C_0(P\times\bar U)$, either the associated 1-dim cocycle $c(t,p)$ satisfies the equivalent conditions in Proposition~\ref{prop-gott} if $h\in B(P\times\bar U)$, or Theorem~\ref{teor-johnson} applies if $h\in \mathcal{U}(P\times \bar U)$. Note also that if the flow on $P$ is periodic, then $C_0(P\times\bar U)=B(P\times\bar U)$.
\end{nota}
For the sake of completeness, and because it will be used later on, we include here a result for 1-dim linear cocycles in line with the Corollary of Theorem~1 in Shneiberg~\cite{shne} given for integrable maps $f:P\to\R$ with zero mean, which says that for almost all $p\in P$ there exists a sequence $(t_n)_n\uparrow \infty$ such that  $\int_0^{t_n} f(p{\cdot}s)\,ds = 0$ for any $n\geq 1$. The corresponding adaptation for cocycles reads as follows.
\begin{teor}\label{teor-shneiberg}
Let $c(t,p)$ be a continuous positive 1-dim linear cocycle and assume that the associated real linear skew-product flow $\R\times P\times\R\to P\times\R$, $(t,p,y)\mapsto (p{\cdot}t,c(t,p)\,y)$ does not have an exponential dichotomy. Then, for almost all $p\in P$ there exists a sequence $(t_n)_n\uparrow \infty$ such that $
c(t_n,p)=1$ for any $n\geq 1$.
\end{teor}
In the following result, the dynamics of the linear semiflow $\tau_L$ is described when $h\in B(P\times \bar U)$. Basically, it means bounded orbits, both away from $0$ and above, for strongly positive initial data.
\begin{teor}\label{teor-equivalencias}
Let $h\in C_0(P\times \bar U)$ and let us fix a reference vector $z_0\gg 0$ in $X$. The following statements are equivalent:
\begin{itemize}
\item[(i)] There exist a $p_0\in P$ and constants $c_0, C_0>0$ such that $c_0\,z_0\leq \phi(t,p_0)\,z_0\leq C_0\,z_0$ for any $t\geq 0$.
\item[(ii)] For any $p\in P$ and $z\in X$, $z\gg 0$, there exist constants $c(p,z), C(p,z)>0$ such that $c(p,z)\,z_0\leq \phi(t,p)\,z\leq C(p,z)\,z_0$ for any $t\geq 0$.
\item[(iii)] $h\in B(P\times \bar U)$.
\item[(iv)] For any $p\in P$ there exists a $C(p)>0$ such that $\phi(t,p)\,z_0\leq C(p)\,z_0$ for any $t\geq 0$.
\item[(v)] For any $p\in P$ there exists a $c(p)>0$ such that $c(p)\,z_0\leq \phi(t,p)\,z_0$ for any $t\geq 0$.
\end{itemize}
\end{teor}
\begin{proof}
(i)$\Rightarrow$(ii): Since the trajectory of $(p_0,z_0)$ under $\tau_L$ lies in the order-interval $[c_0\,z_0,C_0\,z_0]$ and the cone is normal, it is bounded, and we can consider the omega-limit set $K=\mathcal{O}(p_0,z_0)$ which is a compact $\tau_L$-invariant set which projects over the whole $P$. Besides, for any $(p,z)\in K$, $c_0\,z_0\leq z\leq C_0\,z_0$. Now, take a $p\in P$ and a $z\in X$ with $z\gg 0$. For $p$ there is a pair $(p,z^*)\in K$ and we can take constants $c_1(p,z), C_1(p,z)>0$ such that $c_1(p,z)\,z^*\leq z\leq  C_1(p,z)\,z^*$. Then, for any $t\geq 0$, $c_0\,c_1(p,z)\,z_0\leq c_1(p,z)\, \phi(t,p)\,z^*\leq  \phi(t,p)\,z\leq  C_1(p,z)\, \phi(t,p)\,z^*\leq C_0\,C_1(p,z)\,z_0$ and it suffices to take $c(p,z)=c_0\,c_1(p,z)$ and $C(p,z)=C_0\,C_1(p,z)$.
\par
(ii)$\Rightarrow$(iii): Let us see that $\sup_{t\geq 0} | \ln c(t,p)|<\infty$ for any $p\in P$. First of all, from the continuity and strong positivity on the compact set $P$ of the map $e$ giving the leading direction in the continuous separation, and the fact that $z_0\gg 0$, one deduces that there exist constants $c_1,\,C_1>0$ such that $c_1\,z_0\leq e(p)\leq C_1\,z_0$ for any $p\in P$. Then, for $p\in P$ and $e(p)\gg 0$, take $c(p), C(p)>0$ given in (ii) such that $c(p)\,z_0\leq  \phi(t,p)\,e(p)=  c(t,p)\,e(p{\cdot}t)\leq C(p)\,z_0$ for any $t\geq 0$. We can then  deduce that the values of  $c(t,p)$  for $t\geq 0$ move between two positive constants. By Proposition~\ref{prop-gott}  we can conclude that  $h\in B(P\times\bar U)$.
\par
(iii)$\Rightarrow$(i), (iii)$\Rightarrow$(iv) and (iii)$\Rightarrow$(v): Using the previous relation $c_1\,z_0\leq e(p{\cdot}t)\leq C_1\,z_0$ for any $p\in P$ and $t\geq 0$ and~\eqref{c}, it is easy to deduce that
\begin{equation}\label{bounds}
\frac{c_1}{C_1} \,  c(t,p)\,z_0\leq  \phi(t,p)\,z_0\leq \frac{C_1}{c_1} \,  c(t,p)\,z_0\,,\quad p\in P,\;t\geq 0\,.
\end{equation}
Since in particular $\sup_{t\geq 0} | \ln c(t,p)|<\infty$ for any $p\in P$, this means that for each $p\in P$, $c(t,p)$ is bounded away from $0$ and bounded above for any $t\geq 0$. From this, it is immediate to conclude.
\par
(iv)$\Rightarrow$(iii) and (v)$\Rightarrow$(iii): Once more, from \eqref{bounds}, for any $p\in P$ the semicocycle $c(t,p)$ is bounded above by a constant if (iv) holds, and is bounded below by a positive constant if (v) holds.   According to Theorem~\ref{teor-johnson} this can only happen if $h\in B(P\times \bar U)$.  The proof is finished.
\end{proof}
As for the dynamics when  $h\in \mathcal{U}(P\times \bar U)$, we state an oscillation result for $\tau_L$ .
\begin{teor}\label{teor-oscilacion}
Let $h\in \mathcal{U}(P\times \bar U)$. Then, there exists an invariant and residual set $P_{\rm{o}}\subset P$ such that for any $p\in P_{\rm{o}}$ there exist sequences $(t_n^1)_n, (t_n^2)_n\uparrow \infty$ depending on $p$, such that for any $z\in X$ with $z\gg 0$ it holds:
\[ \lim_{n\to \infty} \| \phi(t_n^1,p)\,z \|=0\quad\text{ and }\quad \lim_{n\to \infty} \displaystyle\left\|\frac{1}{ \phi(t_n^2,p)\,z} \right\|=0\,.
\]
\end{teor}
\begin{proof}
Let $P_{\rm{o}}\subset P$ be the invariant and residual set determined in Theorem~\ref{teor-johnson} for the associated real cocycle $c(t,p)$.
Then, for each  $p\in P_{\rm{o}}$ there exist sequences $(t_n^1)_n, (t_n^2)_n\uparrow \infty$ depending on $p$ such that
\[
\lim_{n\to\infty} c(t_n^1,p) =0 \quad\text{and}\quad
\lim_{n\to\infty} c(t_n^2,p)  =\infty\,.
\]
\par
By the properties of $e(p)$, given $z\gg 0$, we can find constants $c_1, C_1>0$ such that $c_1\, e(p)\leq z\leq C_1\, e(p)$ for any $p\in P$.
Then, by relation~\eqref{c}, monotonicity of $\tau_L$ and monotonicity of the norm we get that $\| \phi(t_n^1,p)\,z \|\leq C_1 \,\| \phi(t_n^1,p)\,e(p) \|= C_1 \, c(t_n^1,p)\to 0$ as $n\to\infty$, and
\[
\displaystyle\left\|\frac{1}{\phi(t_n^2,p)\,z} \right\|\leq \frac{1}{c_1} \displaystyle\left\|\frac{1}{ \phi(t_n^2,p)\,e(p)} \right\|= \frac{1}{c_1 \,c(t_n^2,p)} \to 0 \quad\text{as } n\to\infty\,,
\]
as we wanted  to see.
\end{proof}
\par
Now, for $h\in C_0(P\times\bar U)$ we prove the existence of an invariant compact set in $P\times  X$, with a precise dynamical description depending on whether $h\in B(P\times \bar U)$ or $h\in \mathcal{U}(P\times \bar U)$. First, we give a definition. The operator $A$ below is the one defined in Section~\ref{sec-mild solutions}.
\begin{defi}\label{def solucion entera}
A solution $v:I\to X$ of the abstract equation
\begin{equation}\label{abstract eq}
v'(t)  =
 A\, v(t)+\tilde h(p{\cdot}t)\,v(t)\,,\quad t\in I\,,
\end{equation}
along the orbit of $p\in P$ is said to be an {\em entire solution} provided that $I=(-\infty,\infty)$. In that case, $v(t+s)=\phi(t,p{\cdot}s)\,v(s)$ for any $t\geq 0$ and $s\in \R$. An entire solution $v:(-\infty,\infty)\to X$ is said to be {\em negatively bounded} if $\{v(t)\mid t\leq 0\}\subset X$ is a bounded set.
\end{defi}
\begin{prop}\label{prop-compacto invariante}
Let $h\in C_0(P\times\bar U)$. Then,
the following items hold:
\begin{itemize}
\item[(i)] If $v:\R\to X$ is a negatively bounded solution of the abstract equation \eqref{abstract eq} along the orbit of $p_0\in P$, then $v(t)\in X_1(p_0{\cdot} t)$ for any $t\in \R$.
\item[(ii)] If $h\in B(P\times \bar U)$, then there exists a continuous map $\widehat e: P \to \Int X_+$ such that $\widehat e(p)\in X_1(p)$ for any $p\in P$ and  $\widehat e(p{\cdot} t)=\phi(t,p)\,\widehat e(p)$ for any $p\in P$ and $t\geq 0$. Besides, $K=\{(p,\widehat e(p))\mid p\in P\}$ is a minimal set which is a copy of the base $P$ and it is contained in $P\times \Int X_+$.
\item[(iii)]  If $h\in \mathcal{U}(P\times \bar U)$,  there exist a $p_0\in P$ and a bounded entire solution  $v:\R\to X_+\setminus \{0\}$ of~\eqref{abstract eq} along the orbit of $p_0$ such that $K_0=\cls \{(p_0{\cdot}t, v(t))\mid t\in \R\}$ is a pinched $\tau_L$-invariant compact set in $P\times (\Int X_+\cup \{0\})$.
\end{itemize}
\end{prop}
\begin{proof}
By the invariance of the 1-dim principal bundle, to prove (i) it suffices to check that $v(0)\in X_1(p_0)$. The  continuous variation with respect to $p\in P$ of the projections $\Pi_{1,p}:X\to X_1(p)$,  $\Pi_{2,p}:X\to X_2(p)$ implies that there exist $\rho_1,\, \rho_2>0$ such that $\|\Pi_{1,p}\|\leq \rho_1$ and $\|\Pi_{2,p}\|\leq \rho_2$ for any $p\in P$. Let $r=\sup\{\|v(t)\|\mid t\leq 0\}$. If we write $v(t)=z_1(t)+z_2(t)\in X_1(p_0{\cdot}t)\oplus X_2(p_0{\cdot}t)$ for any $t\in \R$, then $\|z_1(t)\|\leq \rho_1 r$ and $\|z_2(t)\|\leq \rho_2 r$ for any $t\leq 0$. Besides, $\phi(t,p_0{\cdot}(-t))\,z_1(-t)=z_1(0)$ and $\phi(t,p_0{\cdot}(-t))\,z_2(-t)=z_2(0)$ for any $t\geq 0$. Then, applying property~(5) in the definition of  continuous separation,
\[
\|z_2(0)\|=\|\phi(t,p_0{\cdot}(-t))\,z_2(-t)\|\leq \|z_2(-t)\|\,M\,e^{-\delta\,t} \|\phi(t,p_0{\cdot}(-t))\,e(p_0{\cdot}(-t))\|\,.
\]
Since $\Sigma_{\text{pr}}(L)=\{\lambda_P(h)\}=\{0\}$, given $0<\lambda< \delta$, there is an exponential dichotomy with full stable subspace for the 1-dim semiflow $e^{-\lambda t}\,\phi(t,p) |_{X_1}$, that is, given $\varepsilon>0$ there exists a $t_0$ such that $
\|\phi(t,p_0{\cdot}(-t))\,e(p_0{\cdot}(-t))\|\leq \varepsilon\, e^{\lambda t}$ for any $t\geq t_0$. Therefore, we can easily deduce that $\|z_2(0)\|=0$, so that $v(0)\in X_1(p_0)$.
\par
Now assume that $h\in B(P\times \bar U)$. Then, by Proposition~\ref{prop-gott} there exists a function $k\in C(P)$ such that $k(p{\cdot}t)-k(p)=\ln c(t,p)$ for any  $p\in P$, $t\in\R$,
whose exponential $\kappa:P\to \R_+$, $\kappa(p)=\exp k(p)$ is positive and satisfies
\begin{equation*}
\kappa(p{\cdot}t)=\kappa(p)\,c(t,p)\;\; \text{for any}\;\;p\in P ,\; t\in \R\,.
\end{equation*}
Now define the continuous map $\widehat e: P \to \Int X_+$, $p\mapsto \kappa(p)\,e(p)$.
According with~\eqref{c} and the previous formula, for any $t\geq 0$ and any $p\in P$ this map satisfies
\[
\phi(t,p)\,\widehat e(p)=\kappa(p)\,\phi(t,p)\,e(p) =\kappa(p)\,c(t,p)\,e(p{\cdot}t)=\kappa(p{\cdot}t)\,e(p{\cdot}t)=\widehat e(p{\cdot}t)\,.
\]
In other words, it defines a continuous equilibrium for the linear skew-product semiflow $\tau_L$. As a consequence, the set $K=\{(p,\widehat e(p))\mid p\in P\}$ is a minimal set in $P\times\Int X_+$  with the simplest possible structure,  and (ii) is proved.
\par
Finally, let us assume that $h\in \mathcal{U}(P\times \bar U)$. Then, as explained in Remark~\ref{remark-DE}, the associated real linear skew-product flow $\pi:\R\times P\times\R\to P\times\R$, $\pi(t,p,y)=(p{\cdot}t,c(t,p)\,y)$ does not have an exponential dichotomy. In this case, a result by Selgrade~\cite{selg} says that there exists a $p_0\in P$ for which there is a nonzero bounded orbit, that is, $\{c(t,p_0)\mid t\in\R\}$ is bounded in $\R_+$. In this situation we claim that the set $K=\cls\{(p_0{\cdot}t, c(t,p_0))\mid t\in \R\}$ is an invariant compact set in $P\times \R$ for $\pi$ with a pinched structure (this is a generalization of Lemma~14 in Caraballo et al.~\cite{caloNonl}). The fact that $K$ is invariant and compact is clear. In particular, for any $p\in P$ there is at least one pair $(p,y)\in K$. Now, let $P_{\rm{o}}$ be the oscillation set of $c(t,p)$ given in Theorem~\ref{teor-johnson}. Then, for every $p\in P_{\rm{o}}$ the only pair in $K$ is $(p,0)$, as if there were a pair $(p,y)$ with $y\not=0$, the orbit of $(p,y)$ would remain in $K$ but this cannot happen because of its oscillating behaviour. Since $P$ is minimal and $0$ determines an orbit, in fact $(p,0)\in K$ for any $p\in P$. Finally $(p_0,1)\in K$, so that we have proved that $K$ has a pinched structure.
\par
At this point, the entire solution along the trajectory of $p_0$ defined by  $v:\R\to X_+$, $t\mapsto v(t)= c(t,p_0)\,e(p_0{\cdot}t)$ is bounded and  the set $K_0=\cls\{(p_0{\cdot}t, v(t))\mid t\in \R\}$ is an invariant compact set in $P\times (\Int X_+\cup \{0\})$ which is homeomorphic to $K$ and thus has a pinched structure. The proof is finished.
\end{proof}
For the sake of completeness, we collect the fundamental properties of the set $K_0$ in the third item of the previous proposition.
\begin{coro}
Let $h \in \mathcal{U}(P\times\bar U)$. Then, the pinched invariant compact set $K_0\subset P\times (\Int X_+\cup \{0\})$ given in Proposition~$\ref{prop-compacto invariante}~\rm{(iii)}$  satisfies:
\begin{itemize}
  \item[(a)]  $(p_0,e(p_0))\in K_0$.
  \item[(b)] $(p,0)\in K_0$ for any $p\in P$.
  \item[(c)] $(p,0)$ is the only element in $K_0$ if  $p\in P_{\rm{o}}$, the oscillation set of the associated $1$-dim linear cocycle $c(t,p)$.
\end{itemize}
\end{coro}
After the dynamical description of $\tau_L$, depending on whether the coefficient map $h$ is in $B(P\times \bar U)$ or in   $\mathcal U(P\times \bar U)$, we wonder which the  topological size of these sets is. Before we move on, we make a remark which will let us play with an additional term in the equations, providing a technical tool in this paper. Note that, if $h\in C(P\times\bar U)$ and $k\in C(P)$, then $h+k\in C(P\times\bar U)$ and the linear parabolic problem for  $h+k$ given by
\begin{equation}\label{pdefamily h+k}
\left\{\begin{array}{l} \des\frac{\partial y}{\partial t}  =
 \Delta \, y+h(p{\cdot}t,x)\,y+k(p{\cdot}t)\,y\,,\quad t>0\,,\;\,x\in U, \;\, \text{for each}\; p\in P,\\[.2cm]
By:=\alpha(x)\,y+\des\frac{\partial y}{\partial n} =0\,,\quad  t>0\,,\;\,x\in \partial U,\,
\end{array}\right.
\end{equation}
admits the same treatment as the one developed for~\eqref{pdefamily}. Actually, the linear skew-product semiflow $\wit \tau_L$ associated with~\eqref{pdefamily h+k}, which is given by
\begin{equation*}
\begin{array}{cccl}
\wit \tau_L: & \R_+\times P\times X& \longrightarrow & \hspace{0.3cm}P\times X\\
 & (t,p,z) & \mapsto
 &(p{\cdot}t,\wit \phi(t,p)\,z)=(p{\cdot}t,e^{\int_0^t k(p{\cdot}s)\,ds}\phi(t,p)\,z)\,,
\end{array}
\end{equation*}
admits a continuous separation sharing the principal bundle of $\tau_L$. Clearly, the associated 1-dim linear cocycle $\wit c(t,p)$ satisfying  $\wit \phi(t,p)\,e(p) = \wit c(t,p)\,e(p{\cdot}t)$ for $t\geq 0$, $p\in P$ is given by
\begin{equation}\label{ctilde}
\wit c(t,p)= c(t,p)\,\exp \int_0^t k(p{\cdot}s)\,ds\,.
\end{equation}
\par
At this point, note that given $k\in C(P)$, if we consider the 1-dim linear cocycle
\[
 \wit c(t,p)=\exp \int_0^t k(p{\cdot}s)\,ds\,,\quad t\geq 0\,, \;p\in P,
\]
we can easily build a problem \eqref{pdefamily h+k} for which $ \wit c(t,p)$ is the associated 1-dim cocycle. One just has to consider $h\equiv \gamma_0$ ($\gamma_0\geq 0$) the first eigenvalue of the boundary value  problem~\eqref{bvp} with associated  eigenfunction $e_0\in X$, with $e_0\gg 0$ and $\|e_0\|=1$.  In this case, the principal bundle in the induced linear skew-product semiflow is just given by $e(p)=e_0$ for any $p\in P$.
\par
We next collect some easy facts for the terms of the type $h+k$. Recall that relation \eqref{rep int exponente} gives an integral representation of the upper Lyapunov exponent.
\begin{prop}\label{prop-h+k}
The following items hold:
\begin{itemize}
\item[(i)]  If $h\in C(P\times\bar U)$ and $k\in C(P)$, then $h+k\in C(P\times\bar U)$ and
\[
\lambda_P(h+k)=\lambda_P(h)+\int_P k\,d\nu\,.
\]
\item[(ii)] If $h\in C_0(P\times\bar U)$ and $k\in C_0(P)$, then $h+k\in C_0(P\times\bar U)$.
\item[(iii)] If $h\in B(P\times \bar U)$  and $k\in B(P)$, then $h+k\in B(P\times \bar U)$.
\item[(iv)] If $h\in B(P\times \bar U)$ and $ k\in \mathcal{U}(P)$, or if  $h\in \mathcal{U}(P\times \bar U)$ and $k\in B(P)$,  then $h+k\in  \mathcal{U}(P\times \bar U)$.
\end{itemize}
\end{prop}
\begin{proof}
To prove the formula in (i) we use relations \eqref{rep int exponente} and \eqref{ctilde} to get
\begin{align*}
\lambda_P(h+k)&=\int_P\ln \wit c(1,p)\,d\nu=\int_P\ln  c(1,p)\,d\nu+\int_P\int_0^1 k(p{\cdot}s)\,ds\,d\nu \\&= \lambda_P(h)+ \int_0^1 \int_P k(p{\cdot}s)\,d\nu\,ds=\lambda_P(h)+ \int_P k\,d\nu\,,
\end{align*}
where we have applied Fubini's theorem, and the invariance of the measure $\nu$.
\par
Clearly, (ii) follows from (i). For (iii) and (iv) we once more argue from \eqref{ctilde} which means that  for any $t\in\R$ and any $p\in P$,
 \[
 \ln \wit c(t,p)= \ln c(t,p) +\int_0^t k(p{\cdot}s)\,ds\,.
 \]
The proof is finished.
\end{proof}
The properties of the decomposition $C_0(P)=B(P)\cup\mathcal{U}(P)$ can be transferred to the space $C_0(P\times\bar U)$ in the following sense.
\begin{teor}\label{teor-categoria}
Consider the complete metric space $C_0(P\times\bar U)$. Then $C_0(P\times\bar U)=B(P\times \bar U)\,\cup\, \mathcal{U}(P\times \bar U)$ where the union is disjoint and:
\begin{itemize}
\item[(i)] $B(P\times \bar U)$ is a dense set of first category in $C_0(P\times\bar U)$.
\item[(ii)] $\mathcal{U}(P\times \bar U)$ is a residual set in $C_0(P\times\bar U)$.
\end{itemize}
\end{teor}
\begin{proof}
The union is disjoint by the definition, and (ii) follows from (i). To see that $B(P\times \bar U)$ is of first category, let us fix a $p\in P$ and a vector $z_0\gg 0$ in $X$  and  define the sets
\[
B_n=\left\{ h\in B(P\times \bar U) \;\Big|\;\frac{1}{n}\,z_0 \leq  \phi(t,p)\,z_0\leq n\,z_0 \;\,\text{for any}\,\;t\geq 0 \right\},\quad n\geq 1\,.
\]
By Theorem~\ref{teor-equivalencias} (iv)-(v), given $h\in B(P\times \bar U)$ it is clear that $h\in B_n$ for $n$ large enough. Therefore, $B(P\times \bar U)=\cup_{n=1}^\infty B_n$. Next we check that each $B_n$ is a closed set with an empty interior, so that $B_n$ is a nowhere dense set and we are done.
\par
Let us fix an $n\geq 1$ and consider a sequence $(h_j)_j\subset B_n$ with $h_j\to h_0\in C_0(P\times\bar U)$ as $j\to \infty$. For each $j\geq 0$, let $\phi_j(t,p)$ be the associated linear cocycle for $h_j$. Since for every $j\geq 1$, $\frac{1}{n}\,z_0 \leq \phi_j(t,p)\,z_0\leq n\,z_0$ for any $t\geq 0$, it follows  that also $\frac{1}{n}\,z_0 \leq \phi_0(t,p)\,z_0\leq n\,z_0$ for any $t\geq 0$ (for this convergence result, see the proof of Theorem~\ref{teor-comparacion}). Then Theorem~\ref{teor-equivalencias} asserts that $h_0\in B_n$, and it is closed.
\par
About the empty interior, let us argue by contradiction and let us assume that for some $n_0\geq 1$ there exists an $h_0\in \Int B_{n_0}$. Since $\mathcal{U}(P)$ is dense in $C_0(P)$, there is a sequence $(k_j)_j\subset \mathcal{U}(P)$ with $\|k_j\|\leq 1/j$ for any $j\geq 1$. Note that by Proposition~\ref{prop-h+k}, $h_0+k_j\in \mathcal{U}(P\times\bar U)$ for any $j\geq 1$ and $\lim_{j\to\infty} h_0+k_j=h_0$, which is a contradiction.
\par
Finally, to see that $B(P\times \bar U)$ is dense in $C_0(P\times\bar U)$, once more using a result by Schwartzman~\cite{schw} it suffices to see that any map in $C_0(P\times\bar U)$ which is of class $C^1$ in $U$ and of class $C^{1\!}$ along the orbits in $P$  can be approximated by a sequence of maps in $B(P\times \bar U)$.  So take such a regular map $h$ in $\mathcal{U}(P\times\bar U)$. The advantage is that one can associate a 1-dim  cocycle $c_1(t,p)$ to this $h$ with the same behaviour, referring to boundedness, as that of $c(t,p)$, which is further differentiable. More precisely, note that in principle $c(t,p)=\|\phi(t,p)\,e(p)\|$ might not be always differentiable. However, by fixing a point $x_0\in U$ and taking
\begin{equation*}
z_1(p)=\frac{e(p)}{e(p)(x_0)}\in X\,,\quad p\in P,
\end{equation*}
it is not difficult to check that  $\phi(t,p)\,z_1(p)=c_1(t,p)\,z_1(p{\cdot}t)$ for the positive coefficient
\[
c_1(t,p)=v(t,p,z_1(p))(x_0)\quad \text{for any } p\in P,\; t\geq 0\,,
\]
which defines a 1-dim differentiable linear  cocycle: with the regularity conditions on $h$,
$y(t,x)= v(t,p,z_1(p))(x)$ is a classical solution of the IBV problem given by~\eqref{pdefamily} for $p\in P$ with $y(0,x)=z_1(p)(x)$, $x\in \bar U$. Therefore, the map $a(p):=\left.\frac{d}{dt} \ln c_1(t,p)\right|_{t=0}$ is well defined and continuous on $P$ and
\[
 c_1(t,p)=\exp\int_0^t a(p{\cdot}s)\,ds\,,\quad p\in P,\; t\geq 0\,.
 \]
 \par
Note that the relation between $c(t,p)$ and  $c_1(t,p)$ is given by
\[
c_1(t,p)=\frac{e(p{\cdot}t)(x_0)}{e(p)(x_0)}\,c(t,p)\,,\quad p\in P,\; t\geq 0\,,
\]
and there exist constants $c_0,\,C_0>0$ such that $c_0\leq e(p)(x_0)\leq C_0$ for any $p\in P$.
Besides, as commented in the second to last paragraph in the proof of Proposition~\ref{prop-convex}, for any $p\in P$ there exists the limit
\[
\lambda_P=\lim_{t\to\infty} \frac{\ln\|\phi(t,p)\,z_1(p)\|}{t} =\lim_{t\to\infty} \frac{\ln c_1(t,p)}{t}= \lim_{t\to\infty} \frac{1}{t}\,\int_0^t a(p{\cdot}s)\,ds=\int_P a\,d\nu\,,
\]
where the ergodic theorem of Birkhoff has been applied in the last equality. That is, $a\in C_0(P)$. Now the  density of $B(P)$ in $C_0(P)$ permits us to find a sequence of maps $(k_n)_n\subset C_0(P)$ with  $k_n\to 0$ as $n\to\infty$ and  $a+k_n\in B(P)$ for every $n\geq 1$. Now, for each $n\geq 1$, we take $\wit c_n(t,p)$ the associated 1-dim cocycle for $h+k_n$, which satisfies
 \begin{align*}
 \ln \wit c_n(t,p)&= \ln c(t,p) +\int_0^t k_n(p{\cdot}s)\,ds = \ln \frac{e(p)(x_0)}{e(p{\cdot}t)(x_0)}+\ln c_1(t,p) +\int_0^t k_n(p{\cdot}s)\,ds\\
 &=\ln \frac{e(p)(x_0)}{e(p{\cdot}t)(x_0)}+ \int_0^t a(p{\cdot}s)\,ds +\int_0^t k_n(p{\cdot}s)\,ds \,.
 \end{align*}
Since $c_0\leq e(p)(x_0)\leq C_0$ for any $p\in P$ and $a+k_n\in B(P)$, from this relation it follows that for any $p\in P$, $\sup_{t\in\R}|\ln \wit c_n(t,p)|<\infty$, meaning that $h+k_n\in B(P\times\bar U)$ for every $n\geq 1$. Since $h+k_n\to h$ as $n\to\infty$, we are done. The proof is finished.
\end{proof}
To finish this section, recall that in Proposition~\ref{prop-convex} we have proved the convexity of the upper Lyapunov exponent $\lambda_P(h)$. We now prove that it is strictly convex except for the case when the two maps $h_1$ and $h_2$ differ in a map $k(p)$.
\begin{teor}\label{prop-convex estricta}
Let $h_1,\,h_2 \in C(P\times\bar U)$ be of class $C^1$ in $x\in U$ and of class $C^1$ along the trajectories of $P$. Then, for any $0< r< 1$,  $\lambda_P(r h_1+(1-r)h_2)=r\lambda_P( h_1)+(1-r)\lambda_P(h_2) $ if and only if $\nabla_x h_1= \nabla_x h_2$.
\end{teor}
\begin{proof}
First of all, we know that $\lambda_P(r h_1+(1-r)h_2)\leq r\lambda_P(h_1)+(1-r)\lambda_P(h_2)$ for any $h_1,\,h_2\in C(P\times\bar U)$. Now, if $\nabla_x h_1= \nabla_x h_2$, then $h_1=h_2 + k$ for some $k\in C(P)$. Then, using Proposition~\ref{prop-h+k}~(i) it is easy to check that $\lambda_P(r h_1+(1-r)h_2)= r\lambda_P(h_1)+(1-r)\lambda_P(h_2)$ for any $0< r< 1$. Note that in particular for maps $h_1,\,h_2\in C_0(P\times\bar U)$ this means that  $r h_1+(1-r)h_2\in C_0(P\times\bar U)$ for any $0\leq r\leq 1$.
\par
For the converse, we first prove the result for maps in $B(P\times\bar U)$, then in $C_0(P\times\bar U)$ and finally in the general case.
\par
Assume  that $h_1, h_2\in B(P\times\bar U)$ and $\des\frac{\partial h_1}{\partial x_i}\not=\des\frac{\partial h_2}{\partial x_i}$ for some $i\in\{1,\ldots,m\}$ and let us see that $\lambda_P(r h_1+(1-r)h_2)<0$ for any $0<r<1$.  By Proposition~\ref{prop-compacto invariante} (ii) let $b_1,\,b_2: P\to \Int X_+ $ be continuous equilibria respectively for the linear skew-product semiflows
induced by the problems~\eqref{pdefamily} given by $h_1$ and $h_2$, and recall that they lie in the corresponding principle bundle.
\par
Then, as in the proof of Proposition~\ref{prop-convex}, we consider $b:P\to X$, $p\mapsto b(p)=\exp(r\ln b_1(p)+(1-r)\ln b_2(p))$, which satisfies: it is continuous, $C^{1\!}$ along the orbits in $P$, and for every $p\in P$
the map $\bar b:\R_+\times\bar U\to\R,\;(t,x)\mapsto\bar b(t,x)= b(p{\cdot}t)(x)$
is continuously differentiable,
twice continuously differentiable in $x\in U$ and besides, denoting
\begin{equation*}
b'(p)(x)=\frac{\partial}{\partial t}b(p{\cdot}t)(x)|_{t=0}\,,\quad p\in P,\;\,x\in\bar U,
\end{equation*}
it holds
\begin{equation*}
\left\{\begin{array}{l}  b'(p)(x) \geq
 \Delta \,  b(p)(x)+(r h_1(p,x)+(1-r) h_2(p,x))\,b(p)(x)\,,\quad p\in P,\;\,x\in\bar U,\\[.2cm]
B\bar b:=\alpha(x)\,\bar b+\des\frac{\partial \bar b}{\partial n} =0\,,\quad t>0\,,\;\, x\in \partial U.
\end{array}\right.
\end{equation*}
By Lemma~2.11  in N\'{u}\~{n}ez et al.~\cite{nuos3}, $b(p)$ defines a continuous super-equilibrium for the semiflow associated with the linear family~\eqref{pdefamily} with the term $r h_1+(1-r) h_2$, with associated linear cocycle $\Phi(t,p)$.
Let us now see that $ b(p)$ is a strong super-equilibrium. Once more according to Lemma~2.11 in~\cite{nuos3}, it suffices to find some $p_0\in P$ and $x_0\in U$ for which
\begin{equation}\label{strict}
b'(p_0)(x_0)>
 \Delta \,  b(p_0)(x_0)+(r h_1(p_0,x_0)+(1-r) h_2(p_0,x_0))\,b(p_0)(x_0)\,.
\end{equation}
\par
Now, it is not difficult to check that
\begin{align*}
\nabla_x h_1= \nabla_x h_2 &\Longleftrightarrow h_1= h_2 + k \;\text{ for some }k\in C(P) \\ & \Longleftrightarrow
b_1=\lambda\, b_2 \;\text{ for some positive map }\lambda\in C(P)\\&\Longleftrightarrow \nabla_x \ln b_1= \nabla_x \ln b_2\,.
\end{align*}
Since we are assuming that this is not the case, we deduce that there exists a $p_0\in P$ such that $\nabla_x \ln b_1(p_0) \not= \nabla_x \ln b_2(p_0)$, which means that there exists an $x_0\in U$ such that $\nabla_x \ln b_1(p_0)(x_0) \not= \nabla_x \ln b_2(p_0)(x_0$). That is, for some $i\in\{1,\ldots,m\}$,
\[
 \frac{1}{ b_1(p_0)(x_0)} \frac{\partial b_1(p_0)(x_0)}{\partial x_i} \not=  \frac{1}{ b_2(p_0)(x_0)} \frac{\partial b_2(p_0)(x_0)}{\partial x_i} \,,
\]
and going back to the calculations made in the proof of Proposition~\ref{prop-convex} for $z$, and recalling that $\R\to\R$, $s\mapsto s^2$ is a strictly convex map, we deduce  that~\eqref{strict} holds, so that $b(p)$ is a strong super-equilibrium.
\par
Therefore, by considering any $\beta>0$, by linearity $\beta\, b(p)$ is a strong super-equilibrium too. This forces $\lim_{t\to\infty} \Phi(t,p)\,z=0$ for any $z\gg 0$. By Theorems~\ref{teor-equivalencias} and~\ref{teor-oscilacion} it cannot be $\lambda_P(r h_1+(1-r)h_2)=0$ and consequently $\lambda_P(r h_1+(1-r)h_2)<0$ (see also Sacker and Sell~\cite{sase94}), as we wanted to see.
\par
Next, we consider the case $h_1,h_2\in C_0(P\times\bar U)$. If some of them is not in $B(P\times\bar U)$, for instance $h_1\in \mathcal{U}(P\times\bar U)$, then, as seen in the proof of Theorem~\ref{teor-categoria} one can find a map $k\in C_0(P)$ such that $h_1+k \in B(P\times\bar U)$. But since $\lambda_P(h_1+k)=\lambda_P(h_1)$ and $\nabla_x (h_1+k) = \nabla_x h_1$, we can just replace $h_1$ by $h_1+k$ and apply the previous argument for maps in $B(P\times\bar U)$.
\par
Finally, it remains to deal with regular $h_1,h_2\in C(P\times\bar U)$. In this case we just need to note that   $\lambda_P(h-\lambda_P(h))=0$ for any $h\in C(P\times\bar U)$, so that we fall into the previous case considered. The proof is finished.
\end{proof}
\section{Attractors for non-autonomous parabolic PDEs. The case $\lambda_P=0$}\label{sec-nonlinear}\noindent
In this section we consider a family of scalar linear-dissipative parabolic PDEs over a minimal, uniquely ergodic and aperiodic flow $(P,\theta,\R)$, with Neumann or Robin boundary conditions, given for each $p\in P$ by
\begin{equation}\label{pdefamilynl}
\left\{\begin{array}{l} \des\frac{\partial y}{\partial t}  =
 \Delta \, y+h(p{\cdot}t,x)\,y+g(p{\cdot}t,x,y)\, \\
 \;\;\;\;\;\;= \Delta \, y+ G(p{\cdot}t,x,y)\,,\quad t>0\,,\;\,x\in U,
  \\[.2cm]
By:=\alpha(x)\,y+\des\frac{\partial y}{\partial n} =0\,,\quad  t>0\,,\;\,x\in \partial U,\,
\end{array}\right.
\end{equation}
where $h\in C_0(P\times\bar U)=\{h\in C(P\times\bar U)\mid \lambda_P(h)=0\}$, and the nonlinear term $g:P\times \bar U\times \R\to \R$ is continuous and of class $C^1$ with respect to  $y$ and satisfies the following conditions which render the equations dissipative:
\begin{itemize}
  \item[(c1)] $g(p,x,0)=\des\frac{\partial g}{\partial y}(p,x,0)=0$ for any $p\in  P$ and $x\in \bar U$;
  \item[(c2)] $y\,g(p,x,y)\leq 0$ for any $p\in  P$, $x\in \bar U$, and $y\in \R$;
  \item[(c3)] $g(p,x,-y)=-g(p,x,y)$ for any $p\in  P$, $x\in \bar U$, and $y\in \R$;
  \item[(c4)] there exists an $r_0>0$ such that $g(p,x,y)=0$ if and only if $|y|\leq r_0$;
  \item[(c5)] $\des\lim_{|y|\to\infty}\frac{g(p,x,y)}{y}=-\infty$ uniformly on $P\times\bar U$.
\end{itemize}
\par
Under these hypotheses, the {\it a priori\/} only locally defined skew-product semiflow
\begin{equation*}
\begin{array}{cccl}
 \tau: & \R_+\times P\times X& \longrightarrow & \hspace{0.3cm}P\times X\\
 & (t,p,z) & \mapsto
 &(p{\cdot}t,u(t,p,z))\,,
\end{array}
\end{equation*}
induced by the mild solutions of the associated ACPs (see Section~\ref{sec-mild solutions}) is globally defined because of the boundedness of solutions, and it is  strongly monotone as stated in Proposition~\ref{prop-strong monot}. Recall also that the section semiflow $\tau_t$ is compact for every $t>0$ (once more, see Travis and Webb~\cite{trwe}).
\par
Section 3.1 in Cardoso et al.~\cite{cardoso} is devoted to the existence of attractors for linear-dissipative parabolic PDEs of type~\eqref{pdefamilynl} with a general $h\in C(P\times \bar U)$ and conditions (c1), (c2), (c4) and (c5) for the nonlinear term (condition (c3) has been added here for the sake of simplicity). They prove that there exists an absorbing compact set for the  semiflow, thanks to the presence of the nonlinear dissipative term $g(p,x,y)$ (see Proposition~2 in~\cite{cardoso}), so that  $\tau$ has a global attractor $\A=\cup_{p\in P} \{p\}\times A(p)$ for the sets $A(p)=\{z\in X\mid (p,z)\in \A\}\subset X$, which is formed by bounded entire trajectories. Besides, in~\cite{cardoso} the structure of the attractor is studied in  the cases $\lambda_P<0$ and $\lambda_P>0$, where $\lambda_P$ is  the upper Lyapunov exponent  of the linearized family along the null solution, which is of type~\eqref{pdefamily}:
\begin{equation*}
\left\{\begin{array}{l} \des\frac{\partial y}{\partial t}  =
 \Delta \, y+h(p{\cdot}t,x)\,y\,,\quad t>0\,,\;\,x\in U, \;\, \text{for each}\; p\in P,\\[.2cm]
By:=\alpha(x)\,y+\des\frac{\partial y}{\partial n} =0\,,\quad  t>0\,,\;\,x\in \partial U.
\end{array}\right.
\end{equation*}
\par
In this section we concentrate on the unresolved case $\lambda_P=0$, by assuming that $h\in C_0(P\times\bar U)$; just the case that has been studied in detail in Section~\ref{sec-linear}. We keep the notation used up to now for the linear problem. In particular, $\tau_L$ is the induced linear skew-product semiflow.
\par
Note that, on the one hand, for each fixed $p\in P$, the family of compact sets $\{A(p{\cdot}t)\}_{t\in  \R}$ is the pullback attractor for the process on $X$ defined, for each fixed $p\in P$, by $S_p(t,s)\,z=u(t-s,p{\cdot}s,z)$ for any $z\in X$ and $t\geq s$,
 meaning that:
\begin{itemize}
\item[(i)] it is invariant, i.e.,  $S_p(t,s)\,A(p{\cdot}s)=A(p{\cdot}t)$ for any $t\geq s$;
\item[(ii)] it pullback attracts bounded subsets of $X$, i.e., for any bounded set $B\subset X$,
\[
\lim_{s\to -\infty}{\rm dist}(S_p(t,s)\,B,A(p{\cdot}t))=0 \quad \hbox{for any}\; t\in \R\,;
\]
\item[(iii)] it is the minimal family of closed sets with property (ii).
\end{itemize}
A nice reference for processes and pullback attractors is Carvalho et al.~\cite{calaro}.
\par
On the other hand, the non-autonomous set  $\{A(p)\}_{p\in P}$ is a cocycle attractor for the non-autonomous dynamical system (see Section~\ref{sec-preli}). Besides, taking
\[
a(p)=\inf A(p)\quad\text{ and } \quad b(p)=\sup A(p) \quad\text{for any}\;\,p\in P,
\]
these are semicontinuous equilibria for $\tau$ and
\[
\A\subseteq \bigcup_{p\in P} \{p\}\times [a(p),b(p)]\,.
\]
\par
In our case $a(p)=-b(p)$, because of the odd character of the nonlinear term $g(p,x,y)$ in the $y$ variable assumed in (c3). Therefore, we are only going to concentrate on the properties of $b(p)$, but note that in the general case the properties of $b(p)$ can also be immediately transferred to $a(p)$. Finally, as stated in Proposition~3 in~\cite{cardoso}, the pullback attraction property of the cocycle attractor implies that fixed any $z_0\gg 0$, for  $r>0$ large enough,
\begin{equation}\label{b(p)}
b(p)=\lim_{T\to\infty} u(T,p{\cdot}(-T),r z_0)\quad\text{for any}\;\,p\in P.
\end{equation}
\par
The aim of the next two results is to describe the structure of the global attractor depending on whether the map $h\in C_0(P\times\bar U)$ in the linear part lies in $B(P\times\bar U)$ or in $\mathcal{U}(P\times\bar U)$:  in the first case, roughly speaking, $\A$ is a wide set, whereas in the second case it is a pinched set with a complex dynamical structure. In any case, the sections $A(p)$ of $\A$ are very thin sets in the infinite dimensional space $X$.
\begin{teor}\label{teor-estr atractor caso b}
Let $h \in B(P\times\bar U)$ and let $\widehat e: P \to \Int X_+$ be the continuous map given in  Proposition~$\ref{prop-compacto invariante}$ {\rm (ii)}.
Then, there exists an $r_*>0$ such that
\[
A(p)=\{r\,\widehat e(p)\mid |r|\leq r_*\}\subset X_1(p)\quad\text{for any}\;\,p\in P.
\]
\end{teor}
\begin{proof}
The map $\widehat e$ given in Proposition~\ref{prop-compacto invariante} (ii) defines a  continuous  equilibrium for the linear skew-product semiflow $\tau_L$, so that we have a family of continuous equilibria for the linear problem given by $\widehat e_r(p)=r\,\widehat e(p)$, $p\in P$ for each $r>0$. Due to condition (c4) in the nonlinear term, clearly for $r>0$ small enough $\widehat e_r$  is also a continuous equilibrium for the nonlinear semiflow $\tau$. At this point we define
\[
r_*=\sup\{r>0\mid r\,\widehat e(p) \leq \bar r_0 \,\hbox{ for any }\,p\in P\}\,,
\]
for the map $\bar r_0$ in $X$ identically equal to $r_0$, the constant given in (c4).
\par
Clearly, if $|r|\leq r_*$, $\widehat e_r$ is a continuous equilibrium for  $\tau$. Therefore, the set  $\{r\,\widehat e(p)\mid |r|\leq r_*\}\subseteq A(p)$ for any $p\in P$.
\par
By condition (c2) and Theorem~\ref{teor-comparacion}, we can compare solutions of the linear and the nonlinear problems. In particular, for $r>r_*$, $\widehat e_r$ is a super-equilibrium for $\tau$, that is,
$\widehat e_r(p{\cdot}t)\geq u(t,p,\widehat e_r(p))$ for any $p\in P$ and $t\geq 0$, and since it is no longer an equilibrium, there are a $p_0\in P$ and a time $t_0>0$ such that $\widehat e_r(p_0{\cdot}t_0)> u(t_0,p_0,\widehat e_r(p_0))$. Let us compare solutions  and apply the strong monotonicity of the nonlinear semiflow $\tau$ to get, for $t>0$, $
\widehat e_r(p_0{\cdot}(t+t_0))=\phi(t,p_0{\cdot}t_0)\,\widehat e_r(p_0{\cdot}t_0)\geq u(t,p_0{\cdot}t_0,\widehat e_r(p_0{\cdot}t_0))\gg u(t,p_0{\cdot}t_0,u(t_0,p_0,\widehat e_r(p_0))$,
that is, $\widehat e_r(p_0{\cdot}(t+t_0))\gg u(t+t_0,p_0,\widehat e_r(p_0))$. This implies that $\widehat e_r$ is a strong super-equilibrium (see Novo et al.~\cite{nono2}).
Similar arguments to the ones used in the proof of Proposition~2 in~\cite{cardoso} lead to the fact that, fixed a $z_0\gg 0$, for $r$ large enough,
\[
\lim_{T\to\infty} u(T,p{\cdot}(-T),rz_0)=r_*\,\widehat e(p)\quad\text{for any}\;\,p\in P,
\]
and thus, $b(p)=r_*\,\widehat e(p)$ for any $p\in P$.  Then, $\A$ is an invariant compact set also for the linear semiflow $\tau_L$. Since  $\A$ is composed of bounded  entire  trajectories, by Proposition~\ref{prop-compacto invariante} (i) these entire trajectories lie in the principal bundle, that is, $A(p)\subset X_1(p)$, and consequently $ A(p)=\{r\,\widehat e(p)\mid |r|\leq r_*\}$ for any $p\in P$, as we wanted to prove.
\end{proof}
We remark that in the previous situation, $b(p)$ is a continuous equilibrium.
\par
In the following theorem the subindexes s and f  respectively stand for second and first category sets in the Baire sense. The result can be rephrased by saying that  the presence of a pinched global attractor is generic in $C_0(P\times\bar U)$ (see Theorem~\ref{teor-categoria}).
\begin{teor}\label{teor-estr atractor caso u}
Let $h \in \mathcal{U}(P\times\bar U)$. Then, the global attractor $\A$ is a pinched set. More precisely:
\begin{itemize}
\item[(i)] There exists an invariant residual set $P_{\rm{s}}\subsetneq P$ such that $b(p)=0$ for any $p\in P_{\rm{s}}$. In fact $P_{\rm{s}}$ is the set of continuity points of $b$.
\item[(ii)] The set $P_{\rm{f}}=P\setminus P_{\rm{s}}$ is an invariant dense set of first category and $b(p)\gg 0$ for any $p\in P_{\rm{f}}$.
\end{itemize}
\end{teor}
\begin{proof}
According to Proposition~\ref{prop-compacto invariante} (iii) there exist a $p_0\in P$ and a bounded entire solution  $v:\R\to X_+\setminus \{0\}$ of~\eqref{abstract eq} along the orbit of $p_0$ such that $K_0=\cls \{(p_0{\cdot}t, v(t))\mid t\in \R\}\subset P\times (\Int X_+\cup \{0\})$ is a pinched invariant compact set for the linear semiflow $\tau_L$. Taking $\delta>0$, the set $K_\delta=\{(p,\delta z)\mid (p,z)\in K_0\}\subset P\times (\Int X_+\cup \{0\})$ is still  a  pinched $\tau_L$-invariant compact set, and if $\delta$ is small enough so that $\|z\|\leq r_0$  for any $(p,z)\in K_\delta$ ($r_0$ the one given in condition (c4) for $g$), then $K_\delta$ is also a pinched invariant compact set for the nonlinear semiflow $\tau$. Therefore, $K_\delta\subset \A$ and for every $p\in P$, either $b(p)=0$ or $b(p)\gg 0$.
\par
Let $P_{\rm{s}}$ be the set of continuity points of $b$, which is a residual set. Theorem~7 in~\cite{cardoso} asserts that either there exists a $\lambda_0>0$ such that $b(p)\geq \lambda_0\, e_0$ for every $p\in P$ (where $e_0$ has been taken to be the first eigenfunction for~\eqref{bvp} but it might be any $e_0\gg 0$) or $b(p)=0$ for any $p\in P_{\rm{s}}$. So, assume by contradiction that $b(p)\geq \lambda_0\, e_0$ for every $p\in P$ and take a $p_1\in P_{\rm{o}}$, the set given in Theorem~\ref{teor-oscilacion}. Then, there exists a sequence $(t_n)_n\uparrow \infty$ such that $\lim_{n\to\infty}\|\phi(t_n,p_1)\,r e_0\|=0$ for any $r>0$. But if we take $r>0$ big enough so that $r e_0\geq b(p_1)$, then  using Theorem~\ref{teor-comparacion} to compare solutions of the linear and nonlinear problems, $\phi(t_n,p_1)\,r e_0\geq u(t_n,p_1,r e_0)\geq u(t_n,p_1,b(p_1)) =  b(p_1{\cdot}t_n)\geq \lambda_0\, e_0$ for every $n\geq 1$, which is a contradiction. Therefore, $b(p)=0$ for any $p\in P_{\rm{s}}$, and since the pinched set $K_\delta\subset \A$, also $\A$ is pinched.
\par
Note that if $b(p)=0$ for some $p\in P$,  $p$ is a continuity point for $b$, so that $P_{\rm{s}}=\{p\in P\mid b(p)=0\}$. From here it follows that the set $P_{\rm{s}}$ is invariant, since $b(p)=0$ implies $b(p{\cdot}t)=u(t,p,b(p))=0$ for any $t\geq 0$, because the null map is a solution of the nonlinear problem; and if for a $t>0$ it were $b(p{\cdot}(-t))\gg 0$, it would be $u(t,p{\cdot}(-t),b(p{\cdot}(-t)))=b(p)\gg 0$ by the strong monotonicity. Therefore, it is straightforward that $P_{\rm{f}}=P\setminus P_{\rm{s}}$ is an invariant dense set of first category and $b(p)\gg 0$ for any $p\in P_{\rm{f}}$. The proof is finished.
\end{proof}
From now on, we restrict attention to maps $h\in\mathcal{U}(P\times\bar U)$. For a further study of the dynamics of the global attractor $\A$, the sets $P_{\rm{s}}$ and $P_{\rm{f}}$ given in Theorem~\ref{teor-estr atractor caso u} play a fundamental role. Note that, roughly speaking, if  $\nu(P_{\rm{s}})=1$,  the boundary maps $a(p)$ and $b(p)$ of  $\A$ touch each other over a set of full measure, whereas if $\nu(P_{\rm{f}})=1$ these maps only coincide over a set of null measure. \par
We state a technical result, which characterizes the points in  the sets $P_{\rm{f}}$ and $P_{\rm{s}}$  in terms of the behaviour for negative times  of the 1-dim linear cocycle $c(t,p)$ given in Definition~\ref{defi-cociclo c}. Recall that $e(p)\gg 0$, $p\in P$ are the generators of the principal bundle in the continuous separation of $\tau_L$.
\begin{prop}\label{prop-pasado c}
Let $h \in \mathcal{U}(P\times\bar U)$. Then, for $p\in P$:
\begin{itemize}
  \item[(i)] $p\in P_{\rm{f}}$ if and only if $\des\sup_{t\leq 0} c(t,p)<\infty$;
  \item[(ii)] $p\in P_{\rm{s}}$  if and only if $\des\sup_{t\leq 0} c(t,p)=\infty$.
\end{itemize}
\end{prop}
\begin{proof}
It clearly suffices to prove (i). Assume that   $\sup_{t\leq 0} c(t,p)<\infty$ for a certain $p\in P$. Given $r_0$ in condition (c4) for $g$,  $\delta\,c(t,p)\,e(p{\cdot}t)(x)\leq r_0$ for any $t\leq 0$ and $x\in \bar U$ provided that $\delta>0$ is small enough. Then, the solution $u(t,p,\delta\,e(p))$ coincides with the solution $\phi(t,p)\,\delta\,e(p)=\delta\,c(t,p)\,e(p{\cdot}t)$ of the linear abstract problem for negative time, and it is also bounded in forwards time. That is,  it is an entire bounded solution which then lies in the global attractor, i.e., $u(t,p,\delta\,e(p))\in A(p{\cdot}t)$ for any $t\in \R$. In particular, $0\ll \delta\,e(p)\in A(p)$ which implies  $b(p)\gg 0$, i.e., $p\in P_{\rm{f}}$.
\par
Now assume by contradiction  that $\sup_{t\leq 0} c(t,p)=\infty$ for a certain $p\in P_{\rm{f}}$. Since $\{b(p{\cdot}t)\mid t\in \R\}$ is bounded in $X$, there exists a $C>0$ large enough such that $b(p{\cdot}t)\leq C\,e(p{\cdot}t)$ for any $t\in \R$. Then, by monotonicity, Theorem~\ref{teor-comparacion} and~\eqref{c},
\begin{align*}
0\ll b(p)&=u(t,p{\cdot}(-t),b(p{\cdot}(-t)))\leq u(t,p{\cdot}(-t),C\,e(p{\cdot}(-t)))\\&\leq C\,\phi(t,p{\cdot}(-t))\,e(p{\cdot}(-t))=C\, c(t,p{\cdot}(-t))\,e(p)\,\quad \hbox{for any}\;t>0\,.
\end{align*}
\par
By the linear cocycle property $c(t,p{\cdot}(-t))=1/c(-t,p)$,
and from the hypothesis we can take a sequence $(t_n)_n$ of positive times such that $\lim_{n\to\infty} c(-t_n,p)=\infty$. Then it follows  that $\lim_{n\to\infty}c(t_n,p{\cdot}(-t_n))=0$ and consequently $b(p)=0$, which is a contradiction. The proof is finished.
\end{proof}
\begin{coro}\label{coro-1}
Let $h \in \mathcal{U}(P\times\bar U)$. Then the oscillation set $P_{\rm{o}}$ of the cocycle $c(t,p)$ satisfies $P_{\rm{o}}\subseteq P_{\rm{s}}$.
\end{coro}
\par
When looking at the forwards cocycle $c(t,p)$ for $t\geq 0$, the result is the following.
\begin{prop}\label{prop-Rf o Rs}
Let $h\in\mathcal{U}(P\times\bar U)$ and fix a $z_0\gg 0$ in $X$. Then:
\begin{itemize}
\item[(i)] $\nu(P_{\rm{s}})=1$ $\Leftrightarrow$ $\sup_{t\geq 0} \|\phi(t,p)\,z_0\|=\infty$ for a.e.~$p\in P$ $\Leftrightarrow$ $\des\sup_{t\geq 0} c(t,p)=\infty$ for a.e.~$p\in P$;
\item[(ii)] $\nu(P_{\rm{f}})=1$ $\Leftrightarrow$ $\sup_{t\geq 0} \|\phi(t,p)\,z_0\|<\infty$ for a.e.~$p\in P$     $\Leftrightarrow$ $\des\sup_{t\geq 0} c(t,p)<\infty$ for a.e.~$p\in P$.
\end{itemize}
\end{prop}
\begin{proof}
First of all, note that $\sup_{t\geq 0} \|\phi(t,p)\,z_0\|=\infty$ if and only if $\sup_{t\geq 0} c(t,p)=\infty$: it suffices to recall relation~\eqref{c}, take constants $\lambda_1,\lambda_2>0$ such that $\lambda_1\,e(p)\leq z_0\leq \lambda_2\,e(p)$ and apply the monotonicity of both the semiflow and the norm. Second, note that by the cocycle property  the set $\{p\in P\mid \sup_{t\geq 0} c(t,p)=\infty\}$ is invariant, so that its measure is either full or null. Thus, we just need to prove (i).
\par
So, assume first that  $\nu(P_{\rm{s}})=1$. By Theorem~\ref{teor-shneiberg} for almost every $p\in P$ there exists a sequence $(t_n)_n\uparrow \infty$ such that $c(t_n,p)=1$ for any $n\geq 1$. As a consequence, the set of the so-called {\it recurrent points at $\infty$\/},
\[
\{p\in P\mid \text{there exists a sequence }(t_n)_n\uparrow \infty \text{ such that } \lim_{n\to\infty}c(t_n,p)=1\},
\]
has full measure. An application of Fubini's theorem permits to see that for almost every recurrent point, its orbit is made of recurrent points too, so that  we can consider the invariant set of full measure
\begin{equation*}
P_{\rm{r}}^+=\{p\in P \mid p{\cdot}t \;\text{is recurrent at $\infty$ for every}\;t\in\R \}\,.
\end{equation*}
\par
Now, if we take a $p\in P_{\rm{s}}\cap P_{\rm{r}}^+$, on the one hand, by Proposition~\ref{prop-pasado c} we can take a sequence $(t_n^1)_n\downarrow -\infty$ such that $c(t_n^1,p)\to \infty$ as $n\to\infty$. On the other hand, since for any $n\geq 1$, $p{\cdot}t_n^1$ is recurrent at $\infty$, we can find a sequence $(t_n^2)_n\uparrow \infty$ such that $c(t_n^2-t_n^1,p{\cdot}t_n^1)\to 1$ as $n\to\infty$. Then, by the cocycle property,
\[
c(t_n^2,p)=c(t_n^2-t_n^1,p{\cdot}t_n^1)\,c(t_n^1,p)\to \infty \;\text{ as }\; n\to \infty\,,
\]
so that $\sup_{t\geq 0} c(t,p)=\infty$ for almost every $p$ in $P$.
\par
Conversely, assume that $\sup_{t\geq 0} c(t,p)=\infty$ for almost every $p\in P$ and consider the cocycle $\widehat c(t,p)=c(-t,p)$ ($t\in\R$, $p\in P$) for the time-reversed flow on $P$ given by $\widehat \theta:\R\times P\to P$, $(t,p)\mapsto p{\cdot}(-t)$. Theorem~\ref{teor-shneiberg} applied to this cocycle ensures that for almost every $p\in P$ there exists a sequence $(t_n)_n\uparrow \infty$ such that $c(-t_n,p)=1$ for any $n\geq 1$. As a consequence, the set of the so-called {\it recurrent points at $-\infty$\/}
\[
\{p\in P\mid \text{there exists a sequence }(t_n)_n\uparrow \infty \text{ such that } \lim_{n\to\infty}c(-t_n,p)=1\}
\]
has full measure. At this point, a parallel argument to the former one permits to conclude that for almost every $p\in P$, $\sup_{t\leq 0} c(t,p)=\infty$. In other words, by Proposition~\ref{prop-pasado c}, $\nu(P_{\rm{s}})=1$,  as we wanted to see.
\end{proof}
\subsection{The case $\nu(P_{\rm{f}})=1$: chaotic dynamics in the attractor}
In this section we show the presence of chaos, in a very precise sense, inside the attractor, when  $h\in \mathcal{U}(P\times\bar U)$ is such that $\nu(P_{\rm{f}})=1$.
\par
We remark that this case often occurs. The references Johnson~\cite{john81} and Ortega and Tarallo~\cite{orta} provide examples (based on a previous construction by Anosov~\cite{anos}) of quasi-periodic flows $(P,\theta,\R)$ and maps $k\in \mathcal{U}(P)$ with $\sup_{t\in\R}\int_0^t k(p{\cdot}s)\,ds<\infty$ for almost every $p\in P$. In fact these results are expected to be true in a more general setting. As a consequence, using the methods in Proposition~\ref{prop-h+k} and Proposition~\ref{prop-Rf o Rs}, we can assert that for any  regular map  $h\in C_0(P\times\bar U)$  there exists a map $k_1\in C_0(P)$  such that the map  $h+k_1\in\mathcal{U}(P\times\bar U)$ and it  satisfies  $\nu(P_{\rm{f}})=1$. To see it, recall that for every regular $h\in C_0(P\times\bar U)$ we can find a $k_2\in C_0(P)$ such that  $h+k_2\in B(P\times\bar U)$ (see the proof of Theorem~\ref{teor-categoria}) and just take  $k_1=k_2+k$.
\par
For the reader not familiar with the notion of chaos in the sense of Li and Yorke~\cite{liyo}, we include the definition.
\begin{defi}\label{defi-li yorke}
Let $(K,\sigma,\R)$ be a continuous flow on a compact metric space $(K,d)$. (i) A pair $\{x,y\}\subset K$ is called a {\it Li-Yorke pair\/} if
\[
\liminf_{t\to \infty} d(\sigma_tx,\sigma_ty)=0\;\quad\text{and}\;\quad \limsup_{t\to \infty} d(\sigma_tx,\sigma_ty)>0\,.
\]
\par
(ii) A set $D\subseteq K$ is said to be {\it scrambled\/} if every pair $\{x,y\}\subset D$ with $x\not=y$ is a Li-Yorke pair.
\par
(iii) The flow on $K$ is said to be {\it chaotic in the Li-Yorke sense\/} if there exists an uncountable scrambled set in $K$.
\end{defi}
Some dynamical properties associated to the Li-Yorke chaos and its relation with other notions of chaotic dynamics can be found in Blanchard et al.~\cite{bgkm}.
\par
Note that the  restriction of the skew-product semiflow $\tau$ to the global attractor $\A$ is a continuous flow on a compact metric space. For almost periodic equations the base flow $(P,\theta,\R)$ is almost periodic, and thus it is also distal. Consequently, in that case, if $\{(p_1,z_1),(p_2,z_2)\}\subset P\times X$ is a Li-Yorke pair, necessarily $p_1=p_2$. This motivates the following definition. The subindex ch stands for chaos.
\begin{defi}\label{defi-li yorke in measure}
The global attractor $\A$ is said to be {\it fiber-chaotic in measure in the sense of Li-Yorke\/} if there exists a set $P_{\rm{ch}}\subset P$ of full measure such that for every $p\in P_{\rm{ch}}$, $\{p\}\times A(p)$ is an uncountable scrambled set.
\end{defi}
Note that, if  pairs in $\{p\}\times A(p)$ are to exist, it must be $p\in P_{\rm{f}}$. In other words, $P_{\rm{ch}}\subseteq P_{\rm{f}}$. Recall that we are assuming that $\nu(P_{\rm{f}})=1$ for the   first category set $P_{\rm{f}}$, so that there is a chance for chaos in measure. Note also that this notion is different and complements in some way the notion of residually Li-Yorke chaotic sets analyzed by some authors in the context of skew-product flows; for instance, see  Bjerklov and Johnson~\cite{bjjo} and Huang and Yi~\cite{huyi}.
\par
We now give a technical and fundamental result for our purposes, which is a nontrivial generalization of Theorem~35 in~\cite{caloNonl} to this infinite-dimensional setting. Basically, it says that with full measure the attractor consists of entire bounded trajectories of the linear semiflow $\tau_L$.  The subindex l stands for linear.
\begin{teor}\label{teor-casi siempre zona lineal}
Let $h\in\mathcal{U}(P\times\bar U)$ be such that $\nu(P_{\rm{f}})=1$. For the constant $r_0>0$ given in condition {\rm (c4)}, let $\bar r_0\in X$ be the identically equal to $r_0$ map defined on $\bar U$.
Then, there exists an invariant set of full measure $P_{\rm{l}}\subset P_{\rm{f}}$ such that $0\ll b(p)\leq \bar r_0$ for every $p\in P_{\rm{l}}$.
\end{teor}
\begin{proof}
To see that the invariant set $\{p\in P\mid b(p{\cdot}t)\leq \bar r_0 \;\forall\;t\in \R\}$  has full measure, let us assume by contradiction that its complementary set
\[
D=\{p\in P\mid \text{there exist a } t\in \R \text{ and an } x\in \bar U \text{ such that } b(p{\cdot}t)(x) >  r_0\}
\]
has measure one. Note that by Theorem~\ref{teor-estr atractor caso u}, $D\subset P_{\rm{f}}$ and $b(p)\gg 0$ for any $p\in D$.
\par
Recall that the metric space $X=C(\bar U)$ is separable (in particular it is a second-countable topological space) and the measure $\nu$ is a regular Borel measure. In these conditions, we can apply the general form of the classical Lusin's theorem (for instance, see Feldman~\cite{feld}) to the semicontinuous (thus measurable) function  $b:P\to X$  to affirm that, fixed an $\varepsilon>0$, there exists a continuous map $\wit b:P\to X$  such that $\nu(\{p\in P\mid b(p)=\wit b(p)\})>1-\varepsilon$. Since $\nu$ is regular, we can take  a compact set $E_0\subset D\cap \{p\in P\mid b(p)=\wit b(p)\}$ with $\nu(E_0)>0$.
\par
A standard application of Birkhoff's ergodic theorem to the characteristic function of $E_0$ implies that for almost every $p\in P$ there exists a real sequence $(t_n)_n\uparrow \infty$
such that $p{\cdot}t_n\in E_0$ for every $n\geq 1$. Once more, since $\nu$ is a regular measure, we can take a compact set $E_1$ with $\nu(E_1)>0$ such that
\[
E_1\subset \{p\in E_0\mid \text{there exists a sequence }(t_n)_n\uparrow \infty \text{ with } p{\cdot}t_n\in E_0 \;\forall \,n\geq 1 \}\,.
\]
Finally, consider
\[
E_2= \{p\in E_1\mid \text{there exists a sequence }(s_n)_n\uparrow \infty \text{ with } p{\cdot}s_n\in E_1 \;\forall \,n\geq 1 \}
\]
which, again by Birkhoff's ergodic theorem, has $\nu(E_2)=\nu(E_1)>0$. Since the proof is rather technical, we  continue with a series of statements to make it easier to read.
\par
{\it Statement 1\/}: $E_1\subset D_+$, for the set
\[
D_+= \{p\in P\mid \text{there exist a } t>0 \text{ and an } x\in \bar U \text{ such that } b(p{\cdot}t)(x) >  r_0\}\,.
\]
\par
{\it Proof\/}. As a first step, let us prove that there exists a $T_0>0$ such that for any $p\in E_0$ there exist a $t=t(p)$ with $|t|\leq T_0$ and an $x=x(p)\in\bar U$ with $b(p{\cdot}t)(x) >  r_0$. This follows from a compactness argument: note that for a fixed $p\in E_0$ there exist a $t=t(p)\in \R$ and an $x=x(p)\in \bar U$ such that $b(p{\cdot}t)(x) >  r_0$. Since $b(p{\cdot}t)(x)=u(t,p,b(p))(x)=u(t,p,\wit b(p))(x)$, by continuity, there exists a ball $B(p,\delta(p))$ for an appropriate $\delta(p)>0$ such that for any $\wit p\in B(p,\delta(p))\cap E_0$, also  $u(t,\wit p,\wit b(\wit p))(x)=b(\wit p{\cdot}t)(x) >  r_0$. Then, since $E_0\subset \cup_{p\in E_0} B(p,\delta(p))\cap E_0$, there is a finite covering, say $E_0\subset \cup_{i=1}^N B(p_i,\delta(p_i))\cap E_0$ and it suffices to take $T_0=\max\{|t(p_1)|,\ldots,|t(p_N)|\}$.
\par
Now, to finish, take $p\in E_1$ and let us check that $p\in D_+$. Take $s>T_0$ with $p{\cdot}s\in E_0$ and apply the first step: then, there exist a $t=t(p{\cdot}s)$ with $|t|\leq T_0$ and an $x=x(p{\cdot}s)\in\bar U$ with $b(p{\cdot}(t+s))(x) >  r_0$. Since $t+s>0$, $p\in D_+$ and we are done.
\par
{\it Statement 2\/}: If $p_2\in E_2$ and $p_2{\cdot}s_n\in E_1$, $n\geq 1$ for a sequence  $(s_n)_n\uparrow\infty$, then, $\lim_{n\to\infty} \|\phi(s_n-1,p_2)\,b(p_2)\|=\infty$.
\par
{\it Proof\/}. Argue by contradiction and assume without loss of generality that $\{\phi(s_n-1,p_2)\,b(p_2)\mid n\geq 1\}$ is a bounded set in $X$. Once more arguing as in Proposition~2.4 in Travis and Webb~\cite{trwe}, we obtain that the set $\{\phi(s_n,p_2)\,b(p_2)\mid n\geq 1\}$ is relatively compact: just write $\phi(s_n,p_2)\,b(p_2)=\phi(1,p_2{\cdot}(s_n-1))\,\phi(s_n-1,p_2)\,b(p_2)$.
Thus, the set $\{\phi(s_n,p_2)\,b(p_2)\mid n\geq 1\}$ has at least a limit point. Taking a subsequence if necessary, we can assume that  $p_2{\cdot}s_n\to p_1\in E_1$ and  $\phi(s_n,p_2)\,b(p_2)\to z$ as $n\to\infty$. Since $p_2{\cdot}s_n,\,p_1\in E_1\subset E_0$ for any $n\geq 1$, then $b(p_2{\cdot}s_n)\to b(p_1)\gg 0$. Comparing solutions of the nonlinear and the linear problems, $b(p_2{\cdot}s_n)\leq \phi(s_n,p_2)\,b(p_2)$ for any $n\geq 1$, so that in the limit $0\ll b(p_1)\leq z$.
\par
By Statement~1, for $p_1\in E_1\subset D_+$, there exist a $t_1>0$ and an $x_1\in\bar U$ such that $b(p_1{\cdot}t_1)(x_1) >  r_0$. If we look at the solution $b(p_1{\cdot}t)$, $t\geq 0$ of the nonlinear problem for $p_1$, it lies below the solution $\phi(t,p_1)\,b(p_1)$ of the linear problem with the same initial condition $b(p_1)$. Since at time $t_1$, $b(p_1{\cdot}t_1)(x_1) >  r_0$, the zone where the problem is strictly nonlinear, there must exist a time $t_2>t_1$ such that $b(p_1{\cdot}t_2)< \phi(t_2,p_1)\,b(p_1)$. Now, once more comparing solutions and applying the strong monotonicity of the semiflow, for any $t> t_2$, $b(p_1{\cdot}t)\ll \phi(t,p_1)\,b(p_1)$. Let us fix a time $t_3>0$ such that $b(p_1{\cdot}t_3)\ll \phi(t_3,p_1)\,b(p_1)$.
\par
Let $\gamma_1\geq 1$ be the biggest possible  such that $0\ll \gamma_1\,b(p_1)\leq z$. We can then take a sufficiently close $\gamma_2>\gamma_1$ such that
\[
\gamma_2\,b(p_1{\cdot}t_3)\ll \phi(t_3,p_1)\,\gamma_1\,b(p_1)\,.
\]
\par
Now, since $\lim_{n\to\infty} \phi(s_n+t_3,p_2)\,b(p_2)=\lim_{n\to\infty} \phi(t_3,p_2{\cdot}s_n)\,\phi(s_n,p_2)\,b(p_2)=\phi(t_3,p_1)\,z \geq \phi(t_3,p_1)\,\gamma_1\,b(p_1)\gg \gamma_2\,b(p_1{\cdot}t_3)= \gamma_2\,\lim_{n\to\infty} b(p_2{\cdot}(s_n+t_3))$ (recall that $b$ is an equilibrium for the nonlinear problem),  we deduce that there exists an $n_0\in\N$ such that for any $n\geq n_0$, $\phi(s_n+t_3,p_2)\,b(p_2)\geq \gamma_2\,b(p_2{\cdot}(s_n+t_3))$. For any $n\geq n_0$ such that $s_n>s_{n_0}+t_3$ we write $s_n=r_n+s_{n_0}+t_3$ for a positive $r_n$. Then, $\phi(s_n,p_2)\,b(p_2)=\phi(r_n,p_2{\cdot}(s_{n_0}+t_3))\,\phi(s_{n_0}+t_3,p_2)\,b(p_2)\geq \phi(r_n,p_2{\cdot}(s_{n_0}+t_3))\,\gamma_2\,b(p_2{\cdot}(s_{n_0}+t_3))\geq \gamma_2\,b(p_2{\cdot}(r_n+s_{n_0}+t_3))=\gamma_2\,b(p_2{\cdot}s_n)$. Taking limits as $n\to \infty$, we deduce that $z\geq \gamma_2\,b(p_1)$ with $\gamma_2>\gamma_1$, in contradiction with  the definition of $\gamma_1$. We are done.
\par
{\it Statement 3\/}: For any $p_2\in E_2$, $\sup_{t\geq 0} c(t,p_2)=\infty$.
\par
{\it Proof\/}. Since for $p_2\in E_2$, $b(p_2)\gg 0$, there exists a $\gamma>0$ such that $e(p_2)\geq \gamma\,b(p_2)$, so that by monotonicity  $c(t,p_2)\,e(p_2{\cdot}t)=\phi(t,p_2)\,e(p_2)\geq \gamma\,\phi(t,p_2)\,b(p_2)$ for any $t\geq 0$. The boundedness of $e(p_2{\cdot}t)$ for $t\geq 0$ and Statement~2 then imply that $\sup_{t\geq 0} c(t,p_2)=\infty$, as wanted.
\par
To finish the proof, note that, since $E_2 \subset P_{\rm{f}}$, Statement 3 falls into contradiction with Proposition~\ref{prop-Rf o Rs}. Therefore, the invariant set $\{p\in P\mid b(p{\cdot}t)\leq \bar r_0 \;\forall\;t\in \R\}$  has full measure and it suffices to take the intersection of this set with $P_{\rm{f}}$ to obtain the set $P_{\rm{l}}$ in the statement of the theorem. The proof is finished.
\end{proof}
We can now prove that there is chaos in the global attractor.
\begin{teor}\label{teor-caos}
Let $h\in\mathcal{U}(P\times\bar U)$ be such that $\nu(P_{\rm{f}})=1$. Then, the global attractor $\A$ is fiber-chaotic in measure in the sense of Li-Yorke.
\end{teor}
\begin{proof}
For the set $P_{\rm{l}}\subset P_{\rm{f}}$ given in Theorem~\ref{teor-casi siempre zona lineal}, let us take a compact set $E_0\subset P_{\rm{l}}$ with $\nu(E_0)>0$ such that the restriction $b|_{E_0}$ is continuous (which exists by the generalized Lusin's theorem) and consider the set of full measure
\[
P_{\rm{ch}}= \{p\in P_{\rm{l}} \mid \text{there exists a sequence }(s_n)_n\uparrow \infty \text{ with } p{\cdot}s_n\in E_0 \;\forall \,n\geq 1 \}\,.
\]
\par
Now, take a $p\in P_{\rm{ch}}$, and let us see that any pair of distinct points $(p,z_1),\,(p,z_2)\in \{p\}\times A(p)\subset \{p\}\times [-b(p),b(p)]$ is a Li-Yorke pair. Since $p\in P_{\rm{l}}$, $b(p)\gg 0$ and the orbits $b(p{\cdot}t)$ and  $u(t,p,z_i)$ for $i=1,2$ lie, roughly speaking, in the linear zone of the problem, so that they are entire bounded trajectories for the linear skew-product semiflow, and by Proposition~\ref{prop-compacto invariante} (i) they lie inside the principal bundle. As a consequence $z_1, z_2, b(p) \in X_1(p)$ and there exist distinct  $\lambda_1,\,\lambda_2\in\R$ such that
\[
\|u(t,p,z_2)-u(t,p,z_1)\|=|\lambda_2-\lambda_1|\,\|b(p{\cdot}t)\| \quad\hbox{for any}\; t\geq 0\,.
\]
\par
Then, first take a sequence $(s_n)_n\uparrow \infty$   with $p{\cdot}s_n\in E_0$  for every $n\geq 1$ to obtain that $\limsup_{t\to\infty}\|u(t,p,z_2)-u(t,p,z_1)\|>0$, since $b\gg 0$ over the compact set $E_0$ where $b$ is continuous. Second, we now take a $p_0\in P_{\rm{s}}$ and  a sequence $(t_n)_n\uparrow \infty$   with $p{\cdot}t_n\to p_0$ as $n\to\infty$. By Theorem~\ref{teor-estr atractor caso u}, $\lim_{n\to\infty} b(p{\cdot}t_n)=b(p_0)=0$, so that we can conclude that $\liminf_{t\to\infty}\|u(t,p,z_2)-u(t,p,z_1)\|=0$. The proof is finished.
\end{proof}
\begin{nota} Since the set $P_{\rm{ch}}$ is also invariant, the previous dynamical behaviour can be interpreted in the formulation of  processes, by saying that for every $p\in P_{\rm{ch}}$ the pullback attractor $\{A(p{\cdot}t)\}_{t\in \R}$ for the process in $X$, $S_p(t,s)(\,{\cdot}\,)=u(t-s,p{\cdot}s,\,{\cdot}\,)$ ($t\geq s$), is Li-Yorke chaotic.
\end{nota}
As a consequence of the next result, we can affirm that the closed  set
\[
\F=\bigcup_{p\in P} \{p\}\times[-b(p),b(p)]\subset P\times X
\]
is also fiber-chaotic in measure in the sense of Li-Yorke, meaning that for a subset of $P$ with full measure its sections contain a big uncountable scrambled set. Recall that we have denoted by $\Pi_{1,p}:X\to X_1(p)$,  and $\Pi_{2,p}:X\to X_2(p)$ ($p\in P$) the projections over the subspaces of the continuous separation for $\tau_L$.
\begin{prop}
Let $h\in\mathcal{U}(P\times\bar U)$ be such that $\nu(P_{\rm{f}})=1$.  Consider the complete metric space $\F$ which is positively invariant for $\tau$. Given a pair of distinct points $(p,z_1), (p,z_2)\in \F$ with $p\in P_{\rm{ch}}$, two things can happen:
\begin{itemize}
\item[(i)] either $\Pi_{1,p}(z_1)=\Pi_{1,p}(z_2)$, and then it is an asymptotic pair, meaning that
\[
\lim_{t\to\infty}\|u(t,p,z_2)-u(t,p,z_1)\|=0\,;
\]
\item[(ii)] or $\Pi_{1,p}(z_1)\not=\Pi_{1,p}(z_2)$, and then it is a Li-Yorke pair.
\end{itemize}
\end{prop}
\begin{proof}
First note that since $P_{\rm{ch}}\subset P_{\rm{l}}$ for the set $P_{\rm{l}}$ given in Theorem~\ref{teor-casi siempre zona lineal}, $b(p)\gg 0$ and the semiorbits of the pairs $(p,z_1), (p,z_2)\in \F$ for $\tau$ are actually semiorbits for the linear semiflow $\tau_L$, so that $\|u(t,p,z_2)-u(t,p,z_1)\|=\|\phi(t,p)\,z_2-\phi(t,p)\,z_1\|$, and besides, $b(p)\in X_1(p)$ by Proposition~\ref{prop-compacto invariante} (i).
\par
So, if $\Pi_{1,p}(z_1)=\Pi_{1,p}(z_2)$,  then $\|u(t,p,z_2)-u(t,p,z_1)\|=\|\phi(t,p)\,(\Pi_{2,p}(z_2)-\Pi_{2,p}(z_1))\|\leq M\,e^{-\delta t}\,c(t,p)\,\|\Pi_{2,p}(z_2)-\Pi_{2,p}(z_1)\|$ by using property (5) in the description of the continuous separation of $\tau_L$ and relation~\eqref{c}. Since for $p\in P_{\rm{l}}$ the semiorbit of $r e(p)$ for $r>0$ small enough remains in the bounded zone $\F$, $\sup_{t\geq 0}c(t,p)<\infty$ and consequently the pair is asymptotic.
\par
Finally, if $\Pi_{1,p}(z_1)\not=\Pi_{1,p}(z_2)$, $\lim_{t\to\infty}\|\phi(t,p)\,(\Pi_{2,p}(z_2)-\Pi_{2,p}(z_1))\|=0$ as before; and for $\Pi_{1,p}(z_1),\,\Pi_{1,p}(z_2),\, b(p)\in X_1(p)$ we just argue as in the proof of Theorem~\ref{teor-caos} to conclude that the pair is Li-Yorke chaotic.
\end{proof}
\subsection{A non-autonomous discontinuous pitchfork bifurcation diagram}
The purpose of this final section is to present the  conclusions of the
previous sections of the paper in terms of non-autonomous bifurcation theory.
We want to emphasize that the classical bifurcation patterns can exhibit, in this
non-autonomous framework, ingredients of dynamical complexity which are not
possible in the autonomous models.
\par
Precisely, we look at the one-parametric family ($\gamma\in\R$) of scalar reaction-diffusion problems over a minimal, uniquely ergodic and aperiodic flow $(P,\theta,\R)$, with Neumann or Robin boundary conditions, given for each $p\in P$ by
\begin{equation}\label{bifurcation}
\left\{\begin{array}{l} \des\frac{\partial y}{\partial t}  =
 \Delta \, y+(\gamma+h(p{\cdot}t,x))\,y+g(p{\cdot}t,x,y)\,,\quad t>0\,,\;\,x\in U,
  \\[.2cm]
By:=\alpha(x)\,y+\des\frac{\partial y}{\partial n} =0\,,\quad  t>0\,,\;\,x\in \partial U,\,
\end{array}\right.
\end{equation}
where we assume that $h\in \mathcal{U}(P\times\bar U)$ and $g:P\times \bar U\times \R\to \R$ is continuous, of class $C^1$ with respect to $y$, and it satisfies conditions (c1)-(c5) and also
\begin{itemize}
\item[(c6)] $g(p,x,y)$ is convex in $y$ for $y\leq 0$ and concave in $y$ for $y\geq 0$.
\end{itemize}
\par
For instance, $g$ might just be the map given by
\[
g(p,x,y)=\left\{\begin{array}{lr} k(p,x)\,(y+r_0)^3\,, & y\leq -r_0\\
0\,, & -r_0\leq y \leq r_0\\
-k(p,x)\,(y-r_0)^3\,, & y\geq r_0
\end{array}\right.
\]
for a certain positive map $k\in C(P\times\bar U)$ and the constant $r_0$ in (c4), which provides a non-autonomous version  of the classical Chafee-Infante equation. The autonomous equation was studied by Chafee and Infante~\cite{chin} and some non-autonomous versions of this equation together with  bifurcation problems have also been treated in the literature; for instance, see Carvalho et al.~\cite{calaro12}.
\par
The main result reads as follows. Let us denote by $\A_\gamma$ the global attractor of the corresponding skew-product semiflow $\tau_\gamma$ for the value $\gamma$ of the parameter. Let us also denote by $b_\gamma$ its upper boundary map.
\begin{teor}
The following assertions hold:
\begin{itemize}
\item[(i)] If $\gamma<0$, then $\A_\gamma=P\times \{0\}$ is the global attractor and it is globally exponentially stable.
\item[(ii)] If $\gamma=0$, then the global attractor $\A_0\subseteq \bigcup_{p\in P} \{p\}\times [-b_0(p),b_0(p)]$ is a pinched set which contains a unique minimal set $P\times\{0\}$. Its structure has been described in detail in Theorem~$\ref{teor-estr atractor caso u}$. In particular, if $\nu(P_{\rm{f}})=1$, then $\A_0$ is fiber-chaotic in measure in the sense of Li-Yorke.
\item[(iii)] If $\gamma>0$, then the global attractor $\A_\gamma\subseteq \bigcup_{p\in P} \{p\}\times [-b_\gamma(p),b_\gamma(p)]$ with $b_\gamma(p)\gg 0$ for every $p\in P$ and the maps $\pm b_\gamma$ define continuous equilibria. The copies of the base $K_\gamma^\pm=\{(p,\pm b_\gamma(p))\mid p\in P\}$ are globally exponentially stable minimal sets in $P\times \Int X_\pm$, whereas the trivial minimal set $P\times\{0\}$ is unstable. In addition, $b_0(p)=\lim_{\gamma\to 0^+} b_\gamma(p)$.
\end{itemize}
\end{teor}
\begin{proof}
First of all note that by Proposition~\ref{prop-h+k} (i) the upper Lyapunov exponent of the linearized problem of~\eqref{bifurcation} is $\lambda_P(\gamma+h)=\gamma+\lambda_P(h)=\gamma$, since  $\lambda_P(h)=0$.
\par
With no need of condition (c6), (i) has been proved in Proposition~5 in Cardoso et al.~\cite{cardoso} and (ii) has been proved in Theorems~\ref{teor-estr atractor caso u} and~\ref{teor-caos}. Also it has been mentioned in~\cite{cardoso} that if $\lambda_P(\gamma+h)=\gamma>0$, then the semiflow is uniformly persistent in the interior of both the negative and the positive cones (this follows from Mierczy{\'n}ski and Shen~\cite{mish} or from the general theory developed in Novo et al.~\cite{noos7}), so that $b_\gamma(p)\gg 0$ for every $p\in P$, there exists a global attractor for the restriction of the semiflow to both of these cones, and the trivial minimal set $P\times\{0\}$ is unstable.
\par
The fact that $K_\gamma^\pm=\{(p,\pm b_\gamma(p))\mid p\in P\}$ are globally exponentially stable minimal sets follows, once (c6) is brought into play, from the general theory for monotone and concave skew-product semiflows written by N\'{u}\~{n}ez et al.~\cite{nuos4}. More precisely, the uniform persistence of the semiflow precludes the existence of infinitely many strongly positive minimal sets as the ones in Case A2 of Theorem~3.8 in~\cite{nuos4}, so that only Case A1 of this theorem can hold: there exists exactly one strongly positive minimal set, which is a globally exponentially stable copy of the base (the same for the negative cone). Note that, in particular, the maps $\pm b_\gamma$ define continuous equilibria.
\par
Finally, let us see that $b_0(p)=\lim_{\gamma\to 0^+} b_\gamma(p)$. Fix  $z_0\gg 0$ and $r>0$ and note that if $0\leq\gamma_1\leq \gamma_2$, then by Theorem~\ref{teor-comparacion},  $u_{\gamma_1}(T,p{\cdot}(-T),r z_0)\leq u_{\gamma_2}(T,p{\cdot}(-T),r z_0)$ for any $p\in P$ and $T\geq 0$. Since relation~\eqref{b(p)} holds for $r>0$ large enough, it follows that  $b_{\gamma_1}(p)\leq b_{\gamma_2}(p)$ for any $p\in P$. By monotonicity and compactness of the semiflow, there exists the limit $\lim_{\gamma\to 0^+}b_{\gamma}(p)= b_*(p)$ for each $p\in P$. Since $b_0\leq b_\gamma$ for any $\gamma>0$, it is $b_0(p)\leq b_*(p)$ for each $p\in P$. On the other hand, $b_*(p)$ defines an equilibrium for the problem with $\gamma=0$, and therefore it must be contained in the global attractor $\A_0$, so that $b_*(p)\leq b_0(p)$ for any $p\in P$. Therefore, $b_*(p)= b_0(p)$ for any $p\in P$ and the proof is finished.
\end{proof}
\subsection*{Acknowledgements} The authors would like to thank Prof.~Anthony Quas for providing them with a short  proof of  Theorem~\ref{teor-shneiberg}.

\end{document}